%% file: cubic_growth.tex
\documentclass[11pt]{article}
\usepackage[T1]{fontenc}
\usepackage[margin=1in]{geometry}
\usepackage{amsmath,amssymb,amsthm,microtype,array,needspace}
\usepackage{tikz}
\usepackage[hidelinks]{hyperref}
\newtheorem{theorem}{Theorem}
\newtheorem{lemma}[theorem]{Lemma}

\newcommand{\za}{z_A}
\newcommand{\zl}{z_L}
\title{Density and separation for augmented\\Zarankiewicz numbers}
\author{Nikita Lebedev}
\date{September 2026}
\begin{document}
\maketitle
\begin{abstract}
We study the augmented Zarankiewicz problem, in which disjoint pairs of
cells are added to a binary matrix with no all-one $2\times2$ submatrix.
The pairs must satisfy compatibility conditions, and the objective counts
each original occupied cell and each added pair once. We show that
starting with a maximum $C_4$-free matrix can lower the final optimum,
answering a question of Qi, Cui, and Xu.

Let $\za(m,n)$ be the optimum over all $C_4$-free initial matrices,
and $\zl(m,n)$ the optimum when the initial matrix must have the maximum
number of occupied cells. As $n\to\infty$ with $n\le m=o(n^2)$, we prove
\[
 \za(m,n)-\zl(m,n)\ge\left(\frac1{30}-o(1)\right)mn
\]
and determine the sharp second-order term:
\[
 \za(m,n)=\frac{mn}{3}+\left(\frac1{\sqrt6}+o(1)\right)n\sqrt m.
\]
An explicit construction gives a separation at $m=n=1893$.
We also find a sharp density threshold: when $n\to\infty$ and
$m/n^2\to c>0$, the limited density $\zl(m,n)/(mn)$ tends to $1/3$
if and only if $c\ge1/12$. The proofs combine density and stability estimates,
combinatorial constructions, and an exact polynomial certificate.
\end{abstract}

\section{Introduction and definitions}
Qi, Cui, and Xu introduced augmented Zarankiewicz numbers in connection
with sum-of-squares ranks of biquadratic forms~\cite{framework}.
They ask whether the unrestricted and limited parameters $\za,\zl$
are always equal~\cite[Remark 2.5]{framework}.
Previous work gives complete-graph constructions~\cite{general,k5}
and computational bounds in the incidence family~\cite{incidence}.
L\"ofberg and Qi's upper bound for a weak variant implies
$\zl(\binom n2,n)\le(3/16+o(1))n^3$~\cite[Corollary 7.2]{weakincidence}.
The effect of the chosen maximum base is discussed in
\cite[Appendix B]{general}; we prove a sharp quarter guarantee and a
one-third window (Theorems~\ref{thm:everymaximum} and~\ref{thm:universalpairwindow}).

The combinatorial issue is the interaction between two optimizations.
An occupied cell in the initial matrix contributes directly to the
objective, but it can also prevent later pairs from being admissible.
The restrictions depend on the entire final matrix, so choosing a largest
initial matrix can impose a cost much larger than the number of cells it
adds. We quantify this tradeoff and identify the geometry behind it.
Pairing gives the unrestricted problem a one-third leading density and,
near that density, a stable three-block structure. The column-pair
capacity of a maximum $C_4$-free
base then determines whether such a structure is possible. This links
the augmentation problem to stability and block-design questions and
explains why the shape of the initial base matters alongside its size.

We use the original definitions in~\cite[Section 2]{framework}, pinned
to the cited version. Let $m\ge n\ge2$ and let $z(m,n)$ be the maximum
number of occupied cells in a $C_4$-free $m\times n$ binary matrix.
Choose a $C_4$-free fixed-cell set $E_1\subseteq[m]\times[n]$.
A selected \emph{2-edge} is an unordered pair of distinct cells outside
$E_1$. It is nondegenerate if its two rows and two columns are distinct,
row-degenerate if its row repeats, and column-degenerate if its column repeats.
For a family $E_2$ of selected 2-edges, let $O$ be the union of $E_1$
and all their cells. The conditions are:
\begin{itemize}
\item[(S)] Each cell is used by at most one selected 2-edge.
\item[(C2)] For each selected nondegenerate pair $((r,c),(s,d))$,
the opposite cells $(r,d)$ and $(s,c)$ are not both in $O$.
\item[(C3)] For each selected pair $((r,c),(s,d))$ and all
$x\in[m]\setminus\{r,s\}$ and $y\in[n]\setminus\{c,d\}$,
at least one of $(x,y),(x,c),(x,d),(r,y),(s,y)$ lies outside $O$.
Repeated cells are retained for degenerate pairs.
\end{itemize}
These conditions are tested on the full occupied set $O$, after all pairs
have been added. Condition (C3) says that no occupied cell outside a pair's
rows and columns can complete an occupied rectangle with each of its two
members.
The maximum of $|E_1|+|E_2|$ is $\za(m,n)$.
Imposing $|E_1|=z(m,n)$ gives $\zl(m,n)$.
In particular, the empty fixed base is allowed for $\za$.
The objective counts a selected pair once, although it occupies two cells.
Both parameters extend symmetrically to $m<n$ by transposition.
Throughout, put $M=\binom n2$, $F=|E_1|$, and $T=|E_2|$;
fixed row degrees are denoted by $b_r$.
A row, column, or block is called bare if it receives no selected cell;
its fixed cells remain present.

\paragraph{Three optimization settings.}
For a supplied $C_4$-free base $B$, let $T^*(B)$ be the maximum number
of selected pairs in an admissible augmentation preserving $B$.
The distinction is which bases may be chosen:
\begin{center}
\begin{tabular}{@{}ll@{}}
\textbf{Setting} & \textbf{Choice of fixed base}\\[2pt]
Unrestricted, $\za(m,n)$ & Any $C_4$-free base\\
Limited, $\zl(m,n)$ & Any base with $|B|=z(m,n)$\\
Supplied base, $T^*(B)$ & The given base $B$ is held fixed
\end{tabular}
\end{center}
Thus
\[
 \za(m,n)=\max_{B\;C_4\text{-free}}\bigl(|B|+T^*(B)\bigr),\qquad
 \zl(m,n)=z(m,n)+\max_{B\;\text{maximum}}T^*(B).
\]
A lower bound for $\zl$ needs one suitable maximum base; a universal
maximum-base guarantee must hold for every such base. For a supplied base,
maximizing $|B|+T$ is equivalent to maximizing $T$.

\paragraph{A $2\times2$ example.}
The two stars in $O$ below form one selected pair:
\[
 B=\begin{pmatrix}1&0\\0&0\end{pmatrix},\qquad
 O=\begin{pmatrix}1&\ast\\\ast&0\end{pmatrix},\qquad
 B_{\max}=\begin{pmatrix}1&1\\1&0\end{pmatrix}.
\]
For $B$, we have $F=1$, $T=1$, and objective $2$, although $O$ has three
occupied cells. The pair satisfies (S); one opposite corner is empty,
giving (C2), and there is no outside row or column for a (C3) witness.
Only three cells are available, so $T^*(B)=1$.
Every maximum $2\times2$ base has three fixed cells and leaves room for
no pair. Thus $T^*(B_{\max})=0$ and $\zl(2,2)=3$.
Since $F\le3$ and $F+2T\le4$, every augmentation has objective at most
$3$, hence $\za(2,2)=3$ as well. Thus a larger $T^*(B)$ need not give a
larger value of $|B|+T^*(B)$. Theorem~\ref{thm:separation} gives a genuine
separation between the two optimized parameters in larger dimensions.

\Needspace{15\baselineskip}
\subsection{Separation and the sharp second-order term}
\begin{theorem}[Separation]\label{thm:separation}
If $n\to\infty$ and $n\le m=o(n^2)$, then
\[
 \frac{mn}{4}-o(mn)\le\zl(m,n)\le\frac{3mn}{10}+o(mn),\qquad
 \za(m,n)-\zl(m,n)\ge\frac{mn}{30}-o(mn).
\]
At $N=1893$, the following finite bounds hold:
\[
 \za(N,N)\ge1{,}226{,}667>1{,}223{,}138\ge\zl(N,N).
\]
\end{theorem}

\begin{theorem}[Sharp second-order asymptotic]\label{thm:secondorder}
Along every integer sequence with $n\to\infty$ and $n\le m=o(n^2)$,
\[
 z_A(m,n)=\frac{mn}{3}
       +\left(\frac1{\sqrt6}+o(1)\right)n\sqrt m.
\]
Moreover, every configuration attaining $z_A(m,n)$ has fixed-cell count
$F$ and full occupied-cell count $E$ satisfying
\[
 F=\left(\sqrt{\frac23}+o(1)\right)n\sqrt m,
 \qquad E=\frac{2mn}{3}+o(n\sqrt m).
\]
\end{theorem}

\subsection{Further results and methods}

\paragraph{Overview.}
The following table collects the main estimates by regime. All asymptotic
statements refer to $n\to\infty$ with $m\ge n$; the unrestricted column
follows from Theorem~\ref{thm:unrestricted} and the estimate after it, and
$T^*$ refers to an arbitrary maximum base.
\begin{center}
\small
\setlength{\tabcolsep}{5pt}
\begin{tabular}{@{}>{\raggedright\arraybackslash}p{1.75in}ll@{}}
\textbf{Regime}&\textbf{Unrestricted $\za$}&\textbf{Limited $\zl$}\\[2pt]
$m/n^2\to c\in(0,1/12)$ [Thms.~\ref{thm:threshold},~\ref{thm:everymaximum}]
 &$mn/3+O(m)$
 &$(1/4-o(1))mn\le\zl\le(1/3-\varepsilon(c)+o(1))mn$\\[3pt]
$m\ge n^2/12$ [Thm.~\ref{thm:sharpasymptotic}]
 &$mn/3+O(m)$
 &$mn/3+O(m)$\\[3pt]
every maximum base [Thm.~\ref{thm:everymaximum}]
 &
 &$T^*\ge(1/4-o(1))mn$, sharp as $m/n^2\to\infty$\\[3pt]
$\binom n2\le m\le n^2$, every maximum base [Thm.~\ref{thm:universalpairwindow}]
 &
 &$T^*=mn/3+O(m)$
\end{tabular}
\end{center}

\paragraph{Connection with biquadratic forms.}
For the maximum sum-of-squares rank $\mathrm{BSR}(m,n)$ of biquadratic
forms, Qi, Cui, and Xu prove $\mathrm{BSR}(m,n)\ge\za(m,n)$
\cite[Theorem~2.4]{framework}. Thus our unrestricted lower bounds also
bound this rank from below. The original admissibility conditions are
not always the strongest route to such bounds: along $N=2p$ with $p$
an odd prime, L\"ofberg and Qi obtain
$\mathrm{BSR}(\binom N2,N)\ge(1/4+o(1))N^3$ through a relaxed
second-order parameter~\cite[Theorems~3.3 and~7.5]{weakincidence}, whereas
$\za(\binom N2,N)=N^3/6+O(N^2)$.
Our estimates therefore describe both the lower bounds supplied by the
original augmentation construction and the limits of that construction.
An upper bound for $\za$ is not an upper bound for $\mathrm{BSR}$.
The separation theorem shows that requiring a maximum initial base can
lose a positive fraction of the available combinatorial bound; the
stability theorem identifies the configurations that approach its
unrestricted leading term.

\paragraph{Methods.}
The upper bounds turn a local restriction into a global count.
Condition (C3) prevents the two members of a pair from sharing an open
rectangle neighbor. Sampling pivot cells converts this restriction into
a degree-moment bound (Theorem~\ref{thm:paireddensity}); near equality
forces closeness to the complement of three diagonal blocks
(Theorem~\ref{thm:phasestability}; see Figure~\ref{fig:threeblocks}).
Separation then uses the geometry of
maximum bases to bound each pair's degrees and applies the exact
polynomial certificate of Lemma~\ref{lem:lightmatrix}.
For the quarter guarantee, we instead construct pairs using bare anchors
(Lemma~\ref{lem:bareanchors}). Appendix~\ref{sec:constructions} assembles
maximum bases for the threshold results using known classical values
and design-decomposition theorems.

\paragraph{Organization.}
Sections~\ref{sec:paireddensity}--\ref{sec:generalsquare} develop the
density and stability tools and prove the sharp expansion and separation.
The separation lower bound uses Lemma~\ref{lem:universalquarter} in
Section~\ref{sec:everymaximum}. We then prove the universal quarter
and one-third guarantees and state the limited parameter's threshold.
Section~\ref{sec:openproblems} discusses the remaining density questions.
Appendix~\ref{sec:constructions} proves the threshold results, including
the residue and transition cases. Appendices~\ref{sec:deficitprofile}
and~\ref{sec:squarelower} give the full prescribed-deficit profile and
an additional finite square construction. Appendix~\ref{sec:certificate}
contains the polynomial certificate arithmetic used for separation.

\input{paired_density.tex}
\input{unrestricted.tex}

\input{stability.tex}
\input{second_order.tex}
\input{separation.tex}
\input{universal_lower.tex}
\input{universal_third.tex}
\input{threshold.tex}
\input{open_problems.tex}

\paragraph{Formal companion and ancillary files.}
Lean~4 companions check Theorems~\ref{thm:separation}
and~\ref{thm:secondorder} completely, together with selected further
results; their coverage tables state the exact scope, and the full paper is
not formalized. The Lean sources with their coverage tables, the exact
certificate checker of Lemma~\ref{lem:lightmatrix}, and an exact checker
for an explicit $N=1893$ witness accompany this manuscript as ancillary
files.

\paragraph{AI assistance.}
Generative AI tools assisted with proof exploration, verification code,
and editing. The author is responsible for the paper's entire content.

\appendix
\input{affine_core.tex}
\input{construction_frame.tex}
\input{critical_boundary.tex}
\input{interior.tex}
\input{even_transition.tex}
\input{odd_transition.tex}
\input{high_rows.tex}
\input{uniform_synthesis.tex}

\input{deficit_profile.tex}
\input{square_lower.tex}
\input{polynomial_certificate.tex}

\input{references.tex}
\end{document}

%% file: paired_density.tex
\section{Paired-cell density}\label{sec:density}\label{sec:paireddensity}
The upper bound uses only (C3), without a fixed base. Its key observation
is that the two cells of a selected pair have disjoint open rectangle
neighborhoods. A randomized acceptance rule converts this local
restriction into a bound on the total number of occupied cells.
The construction in Section~\ref{sec:unrestricted} attains the resulting
density $2/3$ asymptotically.

\begin{theorem}\label{thm:paireddensity}
If the $E$ occupied cells of an $m\times n$ binary matrix are partitioned
into pairs satisfying \textup{(C3)}, then
\[
 E\le\frac{2mn}{3}+\frac{2m}{3}+\frac{3n}{2}.
\]
\end{theorem}
\begin{proof}
The case $E=0$ is immediate. Assume $E>0$.
For an occupied cell $p=(r,c)$ define its open rectangle neighborhood
\[
 \mathcal Q_p=\{(s,d)\in O:s\ne r,\ d\ne c,\ (r,d),(s,c)\in O\}.
\]
This relation is symmetric. Paired cells have disjoint
$\mathcal Q$-neighborhoods: a common member would be exactly a
\textup{(C3)} witness, also for a degenerate pair.

Choose a shared list of $2k$ occupied pivots. For each cell, read its
neighborhood-membership indicators in consecutive pairs. Accept the cell
if the first equal indicator pair is $11$, and reject it if that pair
is $00$ or if all pairs are mixed. At most one cell from each selected
pair is accepted. Indeed, at the first equal indicator pair for either
cell, $00$ rejects that cell, while $11$ forces $00$ for its partner.
If neither has an equal pair, both are rejected.
For independent Bernoulli$(t)$ indicators, the acceptance probability is
\[
 \alpha_k(t)=t^2\sum_{j=0}^{k-1}\bigl(2t(1-t)\bigr)^j
 \longrightarrow\beta(t)=\frac{t^2}{t^2+(1-t)^2},
\]
uniformly, with error at most $2^{-k}$.

Write $N_r$ for the occupied columns in row $r$, $d_r=|N_r|$, and
$t_{rs}=|N_r\cap N_s|$. For $d_r>0$, sample the pivots independently
and uniformly from the occupied cells of row $r$.
No cell in that row is accepted. In another row $s$, precisely the
$t_{rs}$ cells over $N_r\cap N_s$ are candidates; if $t_{rs}>0$,
each has membership probability $(t_{rs}-1)/d_r$.
For each candidate its indicators are independent, since the pivots are.
The expected number of accepted cells is at most $E/2$ for every $k$.
Replacing $\alpha_k$ by $\beta$ changes the sum by at most $E2^{-k}$;
letting $k\to\infty$ gives
\begin{equation}\label{eq:openpivot}
 \sum_{\substack{s\ne r\\t_{rs}>0}}
 t_{rs}\beta\bigl((t_{rs}-1)/d_r\bigr)\le E/2.
\end{equation}
For $0\le u\le1$, put
\begin{equation}\label{eq:scalar-slack}
 g(u)=2u\beta(u)-3u+1
 =\frac{(1-u)(2u-1)^2}{u^2+(1-u)^2}\ge0.
\end{equation}
For integers $1\le t\le d$, writing $u=(t-1)/d$, this implies
\[
 t\beta(u)=du\beta(u)+\beta(u)\ge(3t-3-d)/2.
\]
For $t=0$ the same lower bound holds with zero on the left.
Summing in \eqref{eq:openpivot} therefore gives
$\sum_{s\ne r}(3t_{rs}-3-d_r)\le E$. Restoring $t_{rr}=d_r$, and then
summing over rows, gives
\[
 3\sum_s t_{rs}\le E+(m+2)d_r+3(m-1),
\]
\begin{equation}\label{eq:degree-moment}
 3\sum_c e_c^2\le(2m+2)E+3m(m-1),
\end{equation}
where $e_c$ are the column degrees. The row inequality is trivial for
$d_r=0$, and $\sum_{r,s}t_{rs}=\sum_c e_c^2$.
Cauchy--Schwarz yields $E^2\le aE+b$, where
$a=2(m+1)n/3>0$ and $b=m(m-1)n\ge0$. Hence
\[
 E\le a+b/a\le\frac{2mn}{3}+\frac{2n}{3}+\frac{3m}{2}.
\]
For the first inequality, if $E>a$ then $E-a\le b/E\le b/a$;
otherwise it is immediate. Condition (C3) is invariant under transposing
the matrix and its pairs, so transposition proves the theorem.
\end{proof}

Deleting the fixed base of any admissible augmentation preserves
\textup{(C3)} and leaves $2T$ cells fully partitioned into pairs.
Consequently
\begin{equation}\label{eq:generaloriginalupper}
 \zl(m,n)\le\za(m,n)
 \le z(m,n)+\frac{mn}{3}+\frac m3+\frac{3n}{4}.
\end{equation}

\paragraph{Fixed-base counts.}
For a $C_4$-free fixed base, column-pair counting gives
$\sum_r\binom{b_r}{2}\le M$, and Cauchy--Schwarz turns this into
$F(F-m)\le mn(n-1)$. Hence
\begin{equation}\label{eq:classicalupper}
 F\le m+n\sqrt m,
\end{equation}
which is immediate for $F\le m$, and otherwise follows from
$(F-m)^2\le F(F-m)\le mn^2$. Conversely, let $n\le m=o(n^2)$ and, by the
prime number theorem, choose a prime $q\le\sqrt m$ with $q\sim\sqrt m$.
The incidence matrix of the $q^2$ affine points and $q^2$ nonvertical
lines over $\mathbb F_q$ is $C_4$-free with column degree $q$. Retaining
$\min\{n,q^2\}$ columns and padding with zero rows and columns gives a
$C_4$-free $m\times n$ base with $q\min\{n,q^2\}$ cells, and
$q^2/m\le\min\{n,q^2\}/n\le1$ shows that this count is
$(1-o(1))n\sqrt m$. Together with \eqref{eq:classicalupper},
\begin{equation}\label{eq:classicalasymptotic}
 z(m,n)\sim n\sqrt m\qquad(n\le m=o(n^2)).
\end{equation}

Now put $L=M-\sum_r\binom{b_r}{2}\ge0$.
Then
\begin{equation}\label{eq:terminaldeficit}
 m+M-F=L+\sum_r(b_r-1)(b_r-2)/2.
\end{equation}
Thus $F\le m+M$. If $m\ge M$, one row per column pair with singleton
padding attains $z(m,n)=m+M$. In this range, zero left side in
\eqref{eq:terminaldeficit} forces every maximum base to consist of
those $M$ pair rows and $m-M$ singleton rows.

%% file: unrestricted.tex
\section{The unrestricted leading term}\label{sec:unrestricted}
\begin{theorem}\label{thm:unrestricted}
As $\min\{m,n\}\to\infty$,
\[
 \za(m,n)=\frac{mn}{3}+o(mn),
\]
uniformly over the aspect ratio.
\end{theorem}
\begin{proof}[Proof of Theorem~\ref{thm:unrestricted}]
By \eqref{eq:generaloriginalupper} and the classical bound
\eqref{eq:classicalupper},
\[
 \za(m,n)\le\frac{mn}{3}+\frac{4m}{3}+n\sqrt m+\frac{3n}{4}.
\]

For the lower bound use the empty fixed base. Take three row blocks
$R_i$ of size $a=\lfloor m/3\rfloor$ and three column blocks $C_i$
of size $b=\lfloor n/3\rfloor$, leaving the remaining rows and columns
empty. Occupy the six blocks $R_i\times C_j$ with $i\ne j$, as in
Figure~\ref{fig:threeblocks}.
Using the same within-block labels, pair
\[
 ((i,\alpha),(j,\gamma))\quad\hbox{with}\quad
 ((j,\alpha),(i,\gamma)),\qquad i<j.
\]
This partitions all occupied cells into $3ab$ nondegenerate pairs.
Both opposite cells are in empty diagonal blocks.
For a pair between indices $i,j$, a common occupied row of its columns
must lie in the third row block $R_k$, and a common occupied column of
its rows must lie in $C_k$. The required witness cell would be in the
empty block $R_k\times C_k$. Thus \textup{(S), (C2), (C3)} all hold.
Since $3ab\ge mn/3-2(m+n)/3$, the two bounds give
\[
 \frac13-\frac{2}{3m}-\frac{2}{3n}
 \le\frac{\za(m,n)}{mn}
 \le\frac13+\frac{4}{3n}+\frac1{\sqrt m}+\frac{3}{4m}.
\]
Both errors vanish as $\min\{m,n\}\to\infty$.
\end{proof}

Thus $\za(m,n)=mn/3+O(m+n\sqrt m)$ throughout $m\ge n\ge2$.
For every fixed $c>0$, this is $mn/3+O_c(m)$ when $m\ge cn^2$.
\input{three_block_figure.tex}

%% file: three_block_figure.tex
\begin{figure}[htbp]
\centering
\begin{tikzpicture}[scale=.64]
\foreach \x/\y in {0/0,0/2,2/0,2/4,4/2,4/4}
  \fill[gray!20] (\x,\y) rectangle +(2,2);
\draw[step=2] (0,0) grid (6,6);
\foreach \y/\i in {5/1,3/2,1/3} \node[left] at (0,\y) {$R_\i$};
\foreach \x/\i in {1/1,3/2,5/3} \node[above] at (\x,6) {$C_\i$};
\draw[dashed] (3,1) rectangle (5,5);
\draw[dotted,thick] (1,1) rectangle (5,3);
\draw[thick] (3,5)--(1,3);
\foreach \x/\y in {3/5,1/3,5/5,5/3,3/1,1/1}
 \fill (\x,\y) circle (2.5pt);
\node[above left] at (3,5) {$p$};
\node[above left] at (1,3) {$q$};
\node at (5,1) {$\times$};
\node[right] at (5.15,1) {$w$};
\end{tikzpicture}
\caption{The three-block construction. Shaded blocks are occupied;
diagonal blocks are empty. The selected pair $p,q$ has its opposite
corners in empty blocks, satisfying (C2). A (C3) witness would have to
lie at $w\in R_3\times C_3$, also empty. The dashed and dotted
rectangles mark the two sets of corners tested by (C3).}
\label{fig:threeblocks}
\end{figure}
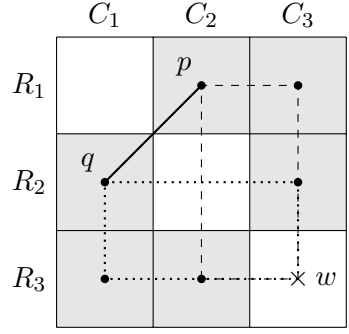

%% file: stability.tex
\section{Stability}
Near the density bound $2/3$, the matrix must resemble the complement of
three diagonal blocks. The proof first controls degrees, then groups rows
with similar supports, and finally recovers the column classes.
The error in the resulting approximation is linear in the density defect.

\begin{theorem}[Linear stability]\label{thm:phasestability}
There are absolute constants $C,\delta_0>0$ with the following property.
Let all $E$ occupied cells of an $m\times n$ binary matrix be partitioned
into (C3)-pairs. If $\epsilon\ge0$, $E/(mn)\ge2/3-\epsilon$, and
$\delta=\epsilon+1/m+1/n\le\delta_0$, there are row and column partitions
into three parts, with $a_i=|R_i|/m$ and $b_i=|C_i|/n$, such that
\[
 \operatorname{dist}_{\rm edit}
 \left(A,\,\overline{\bigcup_i R_i\times C_i}\right)\le C\delta mn,
 \qquad
 \sum_i(a_i-1/3)^2+\sum_i(b_i-1/3)^2\le C\delta.
\]
In particular, if $m,n\to\infty$ and $E/(mn)\to2/3$, the matrix is
within $o(mn)$ cells of this template with all part sizes asymptotic
to one third of their respective dimensions.
\end{theorem}
\begin{proof}
We first show that most rows have degree close to $2n/3$ and that their
normalized codegrees are close to either $1/2$ or $1$. The root $1$
will identify rows with nearly equal supports; the root $1/2$ will
separate three such classes. The last step turns their approximate
column classes into a partition without losing the linear edit bound.

\smallskip\noindent\emph{Degree concentration and codegree slack.}
Put $u_r=d_r/n$, $v_c=e_c/m$, $\alpha=E/(mn)$, and write $t_{rs}$
for row codegrees. The degree moment and its transpose give
\begin{equation}\label{eq:stabilitymoments}
 \frac1m\sum_r(u_r-2/3)^2
 \le\frac{2\epsilon}{3}+\frac1m+\frac{2}{3n}\le\delta,
 \qquad
 \frac1n\sum_c(v_c-2/3)^2\le\delta.
\end{equation}
Indeed, the transpose of \eqref{eq:degree-moment} gives
$m^{-1}\sum_r u_r^2\le(2/3+2/(3n))\alpha+(n-1)/(mn)$.
Subtracting $4\alpha/3$ and adding $4/9$, then using
$\alpha\ge2/3-\epsilon$ and $\alpha\le1$, proves the first estimate;
the other follows in the same way.

Use $g(q)=2q\beta(q)-3q+1$ from \eqref{eq:scalar-slack}.
For $d_r>0$ put $q'_{rs}=(t_{rs}-1)_+/d_r$. Retaining the slack in
\eqref{eq:openpivot} gives
\[
 G':=\sum_{r:d_r>0}\sum_{s\ne r}d_rg(q'_{rs})
 \le(2m+2)E-3\sum_c e_c^2+3m(m-1)
 \le3\delta m^2n.
\]
Indeed $2t\beta((t-1)_+/d)\ge3t-3-d+dg((t-1)_+/d)$;
at $t=0$ the right side is $-3$, so zero codegrees are included.
Cauchy bounds the preceding right side, divided by $m^2n$, by
$2\alpha-3\alpha^2+2\alpha/m+3(m-1)/(mn)$.
Taking $\delta_0\le1/6$ ensures $\alpha\ge1/2$, where
$2\alpha-3\alpha^2$ is decreasing. It is therefore at most
$2\epsilon-3\epsilon^2$, giving the bound $3\delta$.
Since $0\le\beta\le1$ and
$\beta'(q)=2q(1-q)/(q^2+(1-q)^2)^2\le2$, we have $|g'|\le9<20$.
Removing the unit shift therefore costs at most $20m^2$, and the
unshifted diagonal has $g(1)=0$.
Define the symmetric slack
\[
 S_{rs}=u_rg(t_{rs}/d_r)+u_sg(t_{rs}/d_s),
 \qquad \frac1{m^2}\sum_{r,s}S_{rs}\le46\delta,
\]
with zero-degree summands interpreted as zero.
Splitting at $3/4$ in the formula for $g$ shows
\begin{equation}\label{eq:stabilityscalar}
 g(q)\ge\operatorname{dist}(q,\{1/2,1\})^2,
 \qquad g(q)\ge(1-q)/4\quad(3/4\le q\le1).
\end{equation}

\smallskip\noindent\emph{Near classes and their sizes.}
Fix a large absolute $A$. Call a row a representative if
\[
 |u_r-2/3|\le\sqrt{A\delta},\qquad
 \frac1m\sum_sS_{rs}\le A\delta.
\]
Markov's inequality applied to the two nonnegative averages above bounds
the fractions failing these tests by $1/A$ and $46/A$.
Choose $A$ so their sum $47/A$ is less than $1/100$.
Shrink $\delta_0$ so representatives have $u_r\ge3/5$.
For each representative,
\begin{equation}\label{eq:representativeslack}
 \frac1m\sum_s\operatorname{dist}(t_{rs}/d_r,\{1/2,1\})^2
 \le2A\delta.
\end{equation}
Put $\tau=1/20$ and
$L(r)=\{s:|N_r\mathbin\triangle N_s|\le\tau n\}$.
For $s\in L(r)$, each of $d_r-t_{rs}$ and $d_s-t_{rs}$ is at most
$\tau n$, while $d_s\ge d_r-\tau n\ge11n/20$.
Thus both codegree ratios exceed $3/4$, so
\eqref{eq:stabilityscalar} gives
\begin{equation}\label{eq:lineareditnear}
 |N_r\mathbin\triangle N_s|/n\le4S_{rs},\qquad
 \sum_{s\in L(r)}|N_r\mathbin\triangle N_s|\le4A\delta mn.
\end{equation}
Here $|N_r\mathbin\triangle N_s|=d_r+d_s-2t_{rs}$: applying
the second scalar bound to the two summands in $S_{rs}$ controls each
of $d_r-t_{rs}$ and $d_s-t_{rs}$. This linear estimate is what will
keep the final edit error at order $\delta$, rather than $\sqrt\delta$.

Fix $\zeta=1/1000$. For each representative $r$, Markov applied to
\eqref{eq:stabilitymoments} and \eqref{eq:representativeslack} shows
that, apart from $O(\delta)m$ rows $s$, both
$|u_s-2/3|\le\zeta$ and
$\operatorname{dist}(t_{rs}/d_r,\{1/2,1\})\le\zeta$ hold.
Take $\delta_0$ small enough that the representative's degree is also
within $\zeta$ of $2/3$. Write $q=t_{rs}/d_r$. If $q\ge1-\zeta$, then
$|N_r\mathbin\triangle N_s|/n=u_s-u_r+2u_r(1-q)\le4\zeta<\tau$.
If $|q-1/2|\le\zeta$, the same distance is at least
$u_s-2\zeta u_r\ge2/3-3\zeta>\tau$.
Thus the root-one alternative places $s$ in $L(r)$, and the
root-half alternative places it outside. For the row zero sets
$Z_r=[n]\setminus N_r$ and $Z_s=[n]\setminus N_s$, the latter gives
\[
 |Z_r\cap Z_s|/n\le1-u_s-(1/2-\zeta)u_r
       \le13\zeta/6-\zeta^2<3\zeta.
\]
For $\theta=|L(r)|/m$, equation~\eqref{eq:lineareditnear} gives
\[
 \frac1m\sum_{s\in L(r)}\left|\frac{t_{rs}}n-u_r\right|=O(\delta).
\]
Outside $L(r)$, all but $O(\delta)m$ rows take the half-root alternative.
On those rows $|q-1/2|=\operatorname{dist}(q,\{1/2,1\})$; Cauchy and
\eqref{eq:representativeslack} bound its average by $\sqrt{2A\delta}$.
The exceptional rows contribute $O(\delta)$ because all normalized
codegrees lie in $[0,1]$. Consequently
\[
 \frac{\sum_s t_{rs}}{mn}
 =\frac{u_r(1+\theta)}2+O(\sqrt\delta).
\]
On the other hand, interchanging the codegree sum gives
\[
 \frac{\sum_s t_{rs}}{mn}
 =\frac1n\sum_{c\in N_r}v_c
 =\frac{2u_r}{3}+O(\sqrt\delta),
\]
since Cauchy bounds the last error by
$\sqrt{u_r\,n^{-1}\sum_c(v_c-2/3)^2}\le\sqrt\delta$.
Comparing the two expressions and using $u_r\ge3/5$ proves
$|L(r)|/m=1/3+O(\sqrt\delta)$ for every representative.

\smallskip\noindent\emph{Three prototypes cover almost all rows.}
Choose three representatives successively outside the previous near
classes. Also require, from each previously chosen representative's
perspective, codegree distance to $\{1/2,1\}$ at most $B\sqrt\delta$.
Each extra test excludes at most $2A/B^2$ of the rows. Choose $B$ after
$A$, large enough that $4A/B^2<1/100$. At the third choice the two
near classes exclude at most $2/3+O(\sqrt\delta)$ of the rows, and
nonrepresentatives and the two extra tests exclude less than $2/100$.
Thus all three choices are possible after shrinking $\delta_0$.
Their mutual codegrees must be near the half root, since the
one root would put them in the same near class. Take $\delta_0$ small
enough that $\max\{B,\sqrt A\}\sqrt\delta\le\zeta$. For two representatives $r_i,r_j$ the
half root then gives
$|N_{r_i}\mathbin\triangle N_{r_j}|\ge d_{r_j}-2\zeta d_{r_i}\ge(2/3-3\zeta)n>2\tau n$,
so their near classes are disjoint by the Hamming triangle inequality.
Put $Z_i=Z_{r_i}$. The stronger, $B\sqrt\delta$ tests give
$t_{r_i r_j}/n=1/3+O(\sqrt\delta)$, and hence
$|Z_i\cap Z_j|/n=1-u_{r_i}-u_{r_j}+t_{r_i r_j}/n=O(\sqrt\delta)$.
Together with the prototype degrees and inclusion--exclusion, this proves
\begin{equation}\label{eq:prototypezeros}
 |Z_i|=n/3+O(\sqrt\delta)n,\quad
 |Z_i\cap Z_j|=O(\sqrt\delta)n,\quad
 |\bigcup_iZ_i|\ge n-O(\sqrt\delta)n.
\end{equation}

Only $O(\delta)m$ rows lie outside all three near classes.
To see this, discard rows failing either fixed $\zeta$ test for any
prototype, costing $O(\delta)m$ rows. A remaining row outside all near
classes would satisfy $|Z_s\cap Z_i|\le3\zeta n$ for each $i$.
Its intersection with the union of the prototypes is therefore at most
$9\zeta n$, and at most $O(\sqrt\delta)n$ further columns lie outside
that union. By \eqref{eq:prototypezeros} this would give
\[
 (1/3-\zeta)n\le|Z_s|\le9\zeta n+O(\sqrt\delta)n,
\]
which is impossible for small $\delta_0$.
Assign the leftover rows arbitrarily to the three disjoint near classes,
and copy each prototype to its resulting row class, obtaining $A_0$.
The three near classes cost $O(\delta)mn$ edits by
\eqref{eq:lineareditnear}; the leftover rows cost at most $n$ edits each.
Thus
$\operatorname{dist}_{\rm edit}(A,A_0)=O(\delta)mn$ and
$a_i=1/3+O(\sqrt\delta)$.

\smallskip\noindent\emph{Recovering a column partition.}
The sets $Z_i$ need not yet form a partition. A column of $A_0$ has normalized degree
$v_0(c)=1-\sum_{i:c\in Z_i}a_i$.
If it belongs to zero, two or three prototype zero sets, this degree is
within $O(\sqrt\delta)$ of $1$, $1/3$ or $0$, respectively, and hence
at least $1/5$ from $2/3$ for small $\delta_0$.
Only $O(\delta)n$ original columns are at
least $1/10$ from $2/3$, by \eqref{eq:stabilitymoments}.
The other columns with an invalid zero-set membership count have
$|v_0(c)-v_c|\ge1/10$; there are only $O(\delta)n$ of them because
$\sum_c|v_0(c)-v_c|$ is bounded by the edit count divided by $m$.
Assign these exceptional columns to single zero-set labels. This produces
a partition $C_i$ by changing $O(\delta)n$ columns and costs another
$O(\delta)mn$ edits. Each $C_i$ differs from $Z_i$ in $O(\delta)n$ places,
so \eqref{eq:prototypezeros} gives $b_i=1/3+O(\sqrt\delta)$.
The two squared mass bounds follow by summing the six squared
$O(\sqrt\delta)$ deviations. The choices were made in the order
$A$, then $B$, then a sufficiently small $\delta_0$; $\tau$ and $\zeta$
were fixed numerical constants throughout. Thus one absolute choice of
$\delta_0$ suffices for the whole proof.
\end{proof}

\subsection{Reciprocal pairs and exact holes}
\begin{lemma}[Reciprocal matching and clean holes]\label{lem:cleanholes}
Let an admissible configuration have $F$ fixed cells and $T$ selected pairs,
and put
\[
 \eta=(2/3-2T/(mn))_++1/m+1/n+F/(mn).
\]
There are absolute $c,\tau_0>0$ such that, whenever
$0<\tau\le\tau_0$ and $\eta\le c\tau$, there are three-part row and
column partitions, with masses $a_i,b_i$, whose squared deviations from
$1/3$ sum to $O(\eta)$ and whose off-diagonal template differs from full
occupancy in $O(\eta)mn$ cells, and good sets with bad-line fractions
$O(\eta/\tau)$. Good rows and columns have at most $\tau n$ and
$\tau m$ errors against their off-diagonal types. All but
$O(\eta/\tau)mn$ selected pairs are reciprocal, each reciprocal class
has $a_i b_jmn+O(\eta/\tau)mn$ pairs, and every good-by-good diagonal
cell is empty.

If $m,n\to\infty$, $F=o(mn)$ and $2T/(mn)\to2/3$, taking $\tau=\sqrt\eta$
gives asymptotically balanced partitions, uniformly $o(n),o(m)$ good-line
errors, $o(m),o(n)$ bad lines, and $mn/9+o(mn)$ pairs per reciprocal class,
with the same exact diagonal holes.
\end{lemma}
\begin{proof}
We first obtain the template and good lines from stability. We then use
(C3) twice: a large rectangle of possible witnesses rules out
nonreciprocal pairs, and the resulting supply of reciprocal pairs rules
out every occupied good diagonal cell. All neighborhoods in these two
steps refer to the full occupied set, including fixed cells.

\smallskip\noindent\emph{The template and exceptional pairs.}
Delete the fixed cells. The remaining $2T$ cells are partitioned into
pairs, and deletion preserves (C3), so Theorem~\ref{thm:phasestability}
applies with density defect $(2/3-2T/(mn))_+$. Its parameter is at most
$\eta$. Restore the fixed cells, costing at most $F\le\eta mn$ edits.
The full matrix is therefore within $O(\eta)mn$ cells of the template,
with squared class-mass errors $O(\eta)$.

Call a cell of $R_i\times C_j$ a cell of type $(i,j)$.
Mark a row bad if it has more than $\tau n$ discrepancies from its
template row, and a column bad if it has more than $\tau m$ discrepancies.
Markov gives $O(\eta/\tau)m$ bad rows and $O(\eta/\tau)n$ bad columns.
There are also only $O(\eta)mn$ occupied diagonal cells, since every
such cell is an error against the template. Discard from consideration
the selected pairs touching any of these cells or lines. By (S), their
number is at most the number of affected selected cells, hence
$O(\eta/\tau)mn$. This is only a counting exclusion: the underlying
full occupied matrix is unchanged.

\smallskip\noindent\emph{Why the remaining pairs are reciprocal.}
Consider a remaining pair $((r,p),(s,q))$ of types $(i,j),(k,l)$.
Its endpoint rows and columns are good, and $i\ne j$, $k\ne l$.
We seek a witness row class $R_X$ and column class $C_Y$ satisfying
\[
 X\notin\{j,l\},\qquad Y\notin\{i,k\},\qquad X\ne Y.
\]
The first two restrictions make all four required endpoint adjacencies
occupied in the template; the last makes the witness cell itself
off-diagonal. If either excluded set has one element, there are at least
two choices on that side, so distinct $X,Y$ can be chosen. Otherwise
each complement is a singleton. They coincide exactly when
$\{i,k\}=\{j,l\}$, which, because both endpoint types are off-diagonal,
means $(k,l)=(j,i)$. Thus the choice fails precisely for reciprocal pairs.

For a nonreciprocal pair make this choice. Writing $N_O$ for row and
column neighborhoods in full occupancy, put
\[
 W_R=(R_X\cap N_O(p)\cap N_O(q))\setminus\{r,s\},\qquad
 W_C=(C_Y\cap N_O(r)\cap N_O(s))\setminus\{p,q\}.
\]
Goodness of the two endpoint columns and rows gives
\[
 |W_R|\ge(a_X-2\tau)m-2\ge m/5,\qquad
 |W_C|\ge(b_Y-2\tau)n-2\ge n/5.
\]
Here the mass bounds give $a_X,b_Y=1/3+O(\sqrt\eta)$.
Choose $c,\tau_0$ small enough to make these estimates hold; the terms
$1/m+1/n\le\eta$ also control the removal of the actual endpoint
indices. These bounds remain valid if an endpoint row or column repeats.
The rectangle $W_R\times W_C$ lies in an off-diagonal template block
and has at least $mn/25$ cells. For sufficiently small $\eta$, the total
$O(\eta)mn$ edit error cannot empty it. Any occupied cell in this
rectangle supplies all five cells of a (C3) witness, a contradiction.
Every remaining selected pair is therefore reciprocal.

\smallskip\noindent\emph{Counting the reciprocal classes.}
The block $R_i\times C_j$, $i\ne j$, has
$a_i b_jmn+O(\eta)mn$ occupied cells. At most $F\le\eta mn$ of
these are fixed. The exceptional pairs account for at most two selected
cells each in a given block. Every remaining selected cell
in this block belongs to exactly one retained reciprocal pair with the
block $R_j\times C_i$. Hence this class has
$a_i b_jmn+O(\eta/\tau)mn$ retained pairs. Restoring any excluded
reciprocal pairs changes the count by at most the same error, giving
the asserted count for the whole class. Since $\eta/\tau\le c$ and
$a_i,b_j=1/3+O(\sqrt\eta)$, each retained class has at least $mn/20$
pairs after choosing $c,\tau_0$ sufficiently small.

\smallskip\noindent\emph{Exact diagonal holes.}
Suppose a good diagonal cell $(x,y)\in R_h\times C_h$ were occupied,
and let $i,j$ be the other two indices. Use the retained reciprocal
class between types $(i,j)$ and $(j,i)$. For any of its pairs, say
$((r,p),(s,q))$, the four adjacencies needed to make $(x,y)$ a witness
are $(x,p),(x,q),(r,y),(s,y)$; all four are off-diagonal in the template.

The good row $x$ has at most $\tau n$ missing off-diagonal cells.
Each corresponding column contains at most $m$ selected cells, each
belonging to only one pair by (S). Thus these missing adjacencies
exclude at most $\tau mn$ pairs from the class. The at most $\tau m$
missing off-diagonal cells of column $y$ exclude at most $\tau mn$
more. Choose $\tau_0$ also so that $2\tau_0<1/20$.
Fewer than $mn/20$ pairs are excluded, so one remains for which all
four adjacencies are occupied. Its endpoint rows lie in $R_i,R_j$ and
its endpoint columns in $C_j,C_i$, so their actual indices avoid $x,y$.
Together with the assumed occupied cell $(x,y)$, this is a (C3) witness.
The contradiction proves the exact holes, whether the putative diagonal
cell was fixed or selected.

Finally, in the stated asymptotic regime $\eta\to0$.
The choice $\tau=\sqrt\eta$ eventually satisfies both
$\tau\le\tau_0$ and $\eta\le c\tau$. The good-line errors and bad-line
fractions then tend to zero uniformly, and the mass and reciprocal-count
estimates give the qualitative conclusions.
\end{proof}

%% file: second_order.tex
\section{The unrestricted asymptotic expansion}\label{sec:secondorder}
The lower construction inserts a dense $C_4$-free fixed base into the
three-block template. Equal fixed counts in reciprocal blocks mean that
two fixed cells displace one selected pair, increasing the objective by
one. For the upper bound, we estimate full occupancy $E$ and fixed count
$F$ separately, then use $\za=(E+F)/2$ at an optimizer.
Stability supplies the structure needed for these estimates.

\begin{proof}[Proof of Theorem~\ref{thm:secondorder}]
Fix any sequence in the stated regime; all choices of maximizing
configurations below are arbitrary. The finite construction of Lv--Lu--Fang~\cite[Theorem 4]{llf}, for an
odd prime $p$, is a tripartite simple $C_4$-free graph with part size
$a=p(p-1)/2$, in which every vertex has exactly $(p-1)/2$ neighbors
in each other part. Explicitly, with
$R=\{1,\ldots,(p-1)/2\}\subseteq\mathbb F_p$, the parts $V_1,V_2,V_3$ are disjoint copies of $R\times\mathbb F_p$,
and $(a,x)\in V_i$ is joined to
$(b,y)\in V_{i+1}$, indices modulo three, when $y-x=ab$. A variant of this construction
appears in~\cite[Section~2]{llf2}. For fixed $(a,x)$, each $b\in R$ determines
$y$, and each $a'\in R$ determines the $x'$ with $(a',x')\in V_{i-1}$
adjacent to $(a,x)$, which gives the degrees. Two vertices
$(a_1,x_1)\ne(a_2,x_2)$ of one part have at most one common neighbor: two
common neighbors $(b_1,y_1),(b_2,y_2)$ in the same part give
$x_1-x_2=\pm(a_1-a_2)b_j$ for $j=1,2$, hence $a_1\ne a_2$, $b_1=b_2$, and
then $y_1=y_2$; common neighbors $(b,y)\in V_{i+1}$ and $(c,z)\in V_{i-1}$
give $y-z=a_j(b+c)$ for $j=1,2$, and $b+c\ne0$ on $R$ forces $a_1=a_2$ and
$x_1=x_2$. A four-cycle in a tripartite graph has two vertices in one part; since
adjacent vertices have different parts, these vertices are opposite and
have two common neighbors. Thus the graph is $C_4$-free.
Choose a prime $p\le\sqrt{2m/3}$ with
$p/\sqrt{2m/3}\to1$, using the prime number theorem. Then
$m_0=3a\le m$ and $m_0/m\to1$.

Retain all $m_0$ graph vertices as rows of its adjacency matrix, and
choose any $L=\min\{\lfloor n/3\rfloor,a\}$ vertices in each part
as columns. Pad with zero rows and columns to balanced class sizes
$m_i,n_i$ summing to $m,n$. This is possible since $m_i\ge a$ and
$n_i\ge L$. The fixed matrix is $C_4$-free: an adjacency-matrix
rectangle would give a graph $C_4$, with four distinct labels because
the graph has no loops. Counting by retained columns, whose original
neighbors are all retained as rows, gives exactly
\[
 F_{ij}=L(p-1)/2\quad(i\ne j),\qquad F=3L(p-1).
\]
Use the complete off-diagonal template. In each pair of reciprocal
buckets match the available cells to the smaller bucket size, leaving
unmatched cells empty. Condition \textup{(S)} holds and both opposite
cells are holes. A \textup{(C3)} witness in the template would lie in the
third diagonal block; deleting unmatched cells preserves \textup{(C3)}.
With $b_{ij}=m_i n_j-F_{ij}$,
we have $T=\sum_{i<j}\min\{b_{ij},b_{ji}\}$ and
\[
 \begin{split}
 F+T
 &=\frac{mn-\sum_i m_i n_i}{2}+\frac F2
       -\frac12\sum_{i<j}|m_i n_j-m_j n_i|\\
 &=\frac{mn}{3}+\frac F2-O(m+n).
 \end{split}
\]
Here balanced classes give $\sum_i m_i n_i=mn/3+O(1)$, because the
deviations $m_i-m/3$ and $n_i-n/3$ are bounded and sum to zero, so the
cross terms vanish; the fixed-cell counts cancel from the bucket
differences. If
$L=\lfloor n/3\rfloor$, then $3L/n\to1$; otherwise
$m_0<n\le m$ implies the same conclusion. Consequently
$F=(\sqrt{2/3}+o(1))n\sqrt m$. Since
$(m+n)/(n\sqrt m)=\sqrt m/n+1/\sqrt m\to0$, this proves the lower bound.

\smallskip
\noindent\emph{The upper bound: controlling full occupancy.}
Take any maximizing configuration, with $T$ selected
pairs and $E=F+2T$ full occupied cells. The classical $C_4$ bound gives
$F\le m+n\sqrt m=O(n\sqrt m)$. The lower bound and
Theorem~\ref{thm:paireddensity}, applied after deleting fixed cells,
give $2T/(mn)\to2/3$ and $E/(mn)\to2/3$.
Apply the qualitative conclusion of Lemma~\ref{lem:cleanholes} to full
occupancy. Its partitions, good sets and exact diagonal holes require
no maximum-base assumption. Write these partitions as $R_1,R_2,R_3$
and $C_1,C_2,C_3$, with $|R_i|=m/3+o(m)$ and
$|C_i|=n/3+o(n)$. We first show that the fixed cells increase the
accepted-cell count by too little to produce a second-order excess in
$E$. We then bound $F$ by counting column pairs within the three classes.

Call a fixed cell bad if its row or its column is bad, and write
$F_{\rm bad}$ for their number and $F_{\rm good}=F-F_{\rm bad}$.
Then $F_{\rm bad}=o(n\sqrt m)$. Indeed $k=o(m)$ bad
rows contain at most $k+n\sqrt k=o(n\sqrt m)$ fixed cells, since
$k/(n\sqrt m)\le\sqrt m/n\to0$. Likewise $\ell=o(n)$ bad columns
contain at most $m+\ell\sqrt m=o(n\sqrt m)$ fixed cells.

Write $d_r,e_c$ for full degrees, $t_{rs}$ for row codegrees,
$O_{rc}=\mathbf1_{\{(r,c)\in O\}}$, and $x_+=\max\{x,0\}$.
For each nonempty row $r$, sample the finite open-pivot lists used in
Theorem~\ref{thm:paireddensity} from all its occupied cells. Let
$A_{r,k}$ be the expected number of accepted fixed cells at stage $k$.
Condition \textup{(C3)} permits at most one accepted half of each
selected pair, so the
finite inequality is
\[
 \sum_{s\ne r}t_{rs}
   \alpha_k\!\left(\frac{(t_{rs}-1)_+}{d_r}\right)
 \le T+A_{r,k}.
\]
The uniform limit $\alpha_k\to\beta$, where
$\beta(u)=u^2/[u^2+(1-u)^2]$, gives $A_r=\lim_k A_{r,k}$ and
bounds limiting open acceptance by $T+A_r$. Set $A_r=0$ for empty
rows. More explicitly, for $d_r>0$,
\[
 A_r=\sum_{\substack{(s,c)\in E_1\\s\ne r}}
 O_{rc}\beta\!\left(\frac{(t_{rs}-1)_+}{d_r}\right).
\]
Thus the unpaired fixed cells enter through their own acceptance
probabilities. No full-pairing bound is applied to the matrix containing
fixed cells.

A good fixed cell $(s,c)\in R_j\times C_i$ has $j\ne i$, because
every good-by-good diagonal cell is an exact hole.
For a nonempty pivot row $r\ne s$, its open acceptance probability is
$O_{rc}\beta((t_{rs}-1)_+/d_r)$, at most
$O_{rc}\beta(t_{rs}/d_r)$ since $\beta$ is nondecreasing; for $r=s$
it is zero. Good pivot rows in
class $i$ contribute zero by clean holes, and those in class $j$
contribute at most one each. If $r$ is good and belongs to the third
class $h$, its template neighborhood is $C_i\cup C_j$, whereas that of
$s$ is $C_i\cup C_h$. Their intersection is $C_i$. The uniform
$o(n)$ neighborhood errors in Lemma~\ref{lem:cleanholes} therefore give
\[
 d_r=2n/3+o(n),\qquad t_{rs}=|C_i|+o(n)=n/3+o(n).
\]
Consequently $\beta(t_{rs}/d_r)=1/2+o(1)$ uniformly in these rows and
fixed cells. There are $m/3+o(m)$ rows in each class and $o(m)$ bad
pivot rows, each contributing at most one. Summing the class bounds
$0,1,1/2+o(1)$ shows that this cell contributes at most
$m/2+o(m)$ to $\sum_r A_r$.
Bad fixed cells contribute at most $m$ each, and therefore
\begin{equation}\label{eq:fixedacceptsecond}
 \sum_r A_r\le(m/2+o(m))F_{\rm good}+mF_{\rm bad}
 \le mF/2+o(mn\sqrt m).
\end{equation}

The shifted minorant
$2t\beta((t-1)_+/d)\ge3t-d-3$ for integers $d>0$,
$0\le t\le d$, with $t=0$ handled separately, gives
\[
 3\sum_{s=1}^m t_{rs}\le2T+2A_r+(m+2)d_r+3(m-1).
\]
Empty rows satisfy this trivially. Summing and using $2T=E-F$ gives
\[
 3\sum_c e_c^2
 \le(2m+2)E-mF+2\sum_r A_r+3m(m-1).
\]
The $-mF$ term comes from summing $2T=E-F$ over all pivot rows.
Equation~\eqref{eq:fixedacceptsecond} bounds the compensating term
$2\sum_r A_r$ by $mF+o(mn\sqrt m)$, so the fixed-cell contribution
cancels to the required order. The other errors are smaller on the same
scale, since
\[
 \frac{2E}{mn\sqrt m}=O(1/\sqrt m)\to0,\qquad
 \frac{3m(m-1)}{mn\sqrt m}=O(\sqrt m/n)\to0.
\]
Cauchy in the $n$ column degrees gives
$E^2\le n\sum_c e_c^2$, and hence
$3E^2/n\le2mE+o(mn\sqrt m)$. The identity
\[
 3E^2/n-2mE=(3E/n)(E-2mn/3)
\]
has prefactor $3E/n\sim2m$, because $E/(mn)\to2/3$.
Dividing by this positive prefactor yields the one-sided estimate
\begin{equation}\label{eq:fullsecondbound}
 E\le2mn/3+o(n\sqrt m).
\end{equation}

\smallskip
\noindent\emph{Bounding the fixed cells.}
Retain only fixed cells on good rows and columns, leaving
$F_{\rm good}=F-o(n\sqrt m)$ cells. Let $\widehat b_r$ be the retained
degree of row $r$, with $\widehat b_r=0$ for bad rows. A good row in
$R_j$ has no retained cell in $C_j$, by the exact good-by-good diagonal
holes. Its degree therefore splits as $\widehat b_r=u_r+v_r$ between
the other two column classes; set $u_r=v_r=0$ for bad rows.
The number of column pairs within these
classes that occur in this row is
\[
 \binom{u_r}{2}+\binom{v_r}{2}
 =\frac{u_r^2+v_r^2-\widehat b_r}{2}
 \ge\frac{\widehat b_r^2}{4}-\frac{\widehat b_r}{2}.
\]
No column pair can occur in two retained rows: its four fixed cells
would form a $C_4$. All the pairs just counted belong to one of the
three column classes, so their total is at most
\[
 \sum_{i=1}^3\binom{|C_i|}{2}=n^2/6+o(n^2).
\]
Summing the row inequalities and using
$\sum_r\widehat b_r=F_{\rm good}$ gives
\[
 \sum_r\widehat b_r^2
 \le4\sum_{i=1}^3\binom{|C_i|}{2}+2F_{\rm good}
 =\frac23n^2+o(n^2),
\]
where $F_{\rm good}\le F=O(n\sqrt m)=o(n^2)$.
Cauchy over the $m$ rows, including the zero retained degrees, now gives
$F_{\rm good}^2\le m\sum_r\widehat b_r^2$.
Restoring the $o(n\sqrt m)$ discarded fixed cells yields
\[
 F\le\left(\sqrt{2/3}+o(1)\right)n\sqrt m.
\]
Together with \eqref{eq:fullsecondbound} and $z_A=(E+F)/2$, this
proves the matching upper bound, since $\sqrt{2/3}/2=1/\sqrt6$.
Finally, the lower construction bounds the same optimal objective
$z_A$ for every choice of extremizer. Combining it with the upper
bounds just proved gives
\[
 \begin{aligned}
 F=2z_A-E&\ge\bigl(\sqrt{2/3}-o(1)\bigr)n\sqrt m,\\
 E=2z_A-F&\ge2mn/3-o(n\sqrt m).
 \end{aligned}
\]
Thus both estimates are equalities to their stated orders. The sequence
of dimensions and the maximizing configurations were arbitrary, so this
also proves the assertions about every extremizer.
\end{proof}

%% file: separation.tex
\section{Separation below quadratic row growth}\label{sec:generalsquare}
For $n\le m=o(n^2)$, maximum size makes large empty base rectangles
impossible, while (C3) produces an empty rectangle from each selected
pair. Together, these facts prevent both endpoint degree sums from being
large. The polynomial certificate then bounds the number of pairs.
The lower bound is Lemma~\ref{lem:universalquarter} in
Section~\ref{sec:everymaximum}. The finite example uses explicit
empty-rectangle bounds and an elementary moment inequality.

\begin{lemma}\label{lem:subquadraticmixing}
Suppose $n\to\infty$ and $n\le m=o(n^2)$. Uniformly over maximum
$C_4$-free $m\times n$ bases, every empty rectangle has area $o(mn)$.
Moreover, restricting to $n-o(n)$ columns makes the maximum column
degree $o(n)$.
\end{lemma}
\begin{proof}
By \eqref{eq:classicalasymptotic}, $F=z(m,n)\sim n\sqrt m$.
Put $\rho=\sqrt m/n\to0$ and $\tau=\sqrt\rho$.
Set all entries in columns of degree greater than $\tau n$ to zero.
Their number is $k\le F/(\tau n)=O(\tau n)=o(n)$, and they contain
at most $m+k\sqrt m=o(F)$ fixed cells by the elementary $C_4$ bound.
The modified matrix has $f\sim n\sqrt m$ cells and maximum column
degree $\Delta=o(n)$. Every empty rectangle of the original matrix
remains empty in it.

Write $d_r$ for its row degrees and $p=f/m$.
Let $\ell$ count ordered pairs of distinct columns with no common
occupied row. Pair counting gives
\[
 V:=\sum_r(d_r-p)^2\ge0,\qquad
 V+\ell=n(n-1)+f-f^2/m=o(n^2).
\]
Fix an empty $a\times b$ rectangle. For its column set $C$, put
$s_r=|N(r)\cap C|$ and $f_C=\sum_r s_r$.
Counting ordered column pairs, then applying Cauchy, gives
\[
 \sum_r s_r^2\le b^2+b\Delta,\qquad
 \sum_r s_rd_r\ge b(n-1)-\ell,\qquad
 |pf_C-\sum_r s_rd_r|\le\sqrt{V(b^2+b\Delta)}=o(n^2).
\]
If $ab\ge\eta mn$ for a fixed $\eta>0$, then $a\ge\eta m$ and
$b\ge\eta n$, so $pf_C\ge(1-o(1))bn$.
On the other hand, $s_r=0$ on the rectangle's rows, and Cauchy gives
\[
 (pf_C)^2\le p^2(m-a)(b^2+b\Delta)
 \le(1-\eta+o(1))(bn)^2,
\]
a contradiction. Hence every empty rectangle has area $o(mn)$.
All estimates are uniform over maximum bases, since they use only their
common size, the degree threshold, and the bound on removed cells.
\end{proof}

\input{joint_bound.tex}

\begin{proof}[Proof of Theorem~\ref{thm:separation}]
Lemma~\ref{lem:universalquarter} proves the lower bound.
Let $A$ be full occupancy above a maximum $m\times n$ base.
For a selected pair $(x,c),(y,d)$, its common occupied rows outside
$x,y$ and common occupied columns outside $c,d$ form an empty rectangle
by (C3). Its side lengths are at least
\[
 \deg_A(c)+\deg_A(d)-m-2,\qquad
 \deg_A(x)+\deg_A(y)-n-2,
\]
including for degenerate pairs.
Lemma~\ref{lem:subquadraticmixing} gives a uniform $\xi\to0$ such
that an empty base rectangle cannot have both sides greater than
$\xi m$ and $\xi n$. Every pair therefore has normalized full degree
sum at most $1+\xi+2/n$ on one axis.
Delete only the base, retaining the grid and every selected pair.
The degrees can only decrease, and (C3) is preserved. The moment error
$\eta_{m,n}=o(1)$ above is eventually below $1/100$. Thus
Lemma~\ref{lem:lightmatrix} with $\epsilon=\xi+2/n$ gives $2T\le(3/5+o(1))mn$ uniformly.
Since $F\sim n\sqrt m=o(mn)$, this yields
$\zl(m,n)\le3mn/10+o(mn)$. Theorem~\ref{thm:unrestricted}
gives the asymptotic gap.

For the finite example, take the prime $q=43$ and $N=q^2+q+1=1893$.
Plane incidence attains the classical upper bound $F=N(q+1)$.
Equality in the pair count and Cauchy forces every maximizing base to
have row degree $q+1$; applying the same argument to its transpose
gives column degree $q+1$.
For an empty $a\times b$ rectangle with column set $C$, put
$s_r=|N(r)\cap C|$. Regularity and $C_4$-freeness give
\[
 \sum_r s_r=(q+1)b,\qquad
 \sum_r s_r^2\le b(b-1)+(q+1)b=b^2+qb.
\]
The rectangle's $a$ rows contribute zero, so Cauchy yields
\[
 (q+1)^2b^2\le(N-a)(b^2+qb).
\]
For $b>0$, divide by $b$ and use $(q+1)^2=N+q$ to obtain
$ab+q(a+b)\le qN$; the case $b=0$ follows from $a\le N$.
Since
$246^2+86\cdot246=81672>81399=qN$, the integer sides cannot
both exceed $245$. The same C3 argument therefore gives a full row-
or column-degree sum at most $N+245+2$ for every pair.
Deleting the base subtracts $2(q+1)$ from either sum, even for a repeated
endpoint line. For the selected matrix $A$, put $\mu=2T/N^2$, let $u_r,v_c$
be its normalized degrees, and write $\mathbb E$ for the uniform cell
average. Then
\[
 r=\frac{N+245+2-2(q+1)}{2N}=\frac{342}{631},
 \qquad I:=\mathbb E[A_{rc}\min(u_r,v_c)]\le r\mu.
\]
Indeed, the two endpoint minima of each selected pair sum to at most
$2r$, because each minimum is bounded by its degree on the light axis.
Summing over the $T$ disjoint pairs and dividing by $N^2$ gives
$I\le2rT/N^2=r\mu$.
The degree-moment bound \eqref{eq:degree-moment} and its transpose give
\[
 S:=\mathbb E[A_{rc}(u_r+v_c)]
 \le\left(\frac43+\frac4{3N}\right)\mu+\frac{2(N-1)}{N^2}.
\]
Since $\mathbb E[(u_r-v_c)^2]=S-2\mu^2$, apply
$|x|\le2x^2+1/8$ and drop $A$ only from the nonnegative quadratic term to obtain
\[
 I\ge\frac12\left[S-2(S-2\mu^2)-\frac\mu8\right]
    =2\mu^2-\frac S2-\frac\mu{16}.
\]
Combining these and putting
$B=r+2/3+2/(3N)+1/16=115535/90864$ yields
\[
 f(T):=8T^2-2BN^2T-(N-1)N^2\le0.
\]
This quadratic has one positive root, and
$f(1{,}139{,}847)=43{,}849{,}617/8>0$.
Thus $T\le1{,}139{,}846$, and restoring $N(q+1)=83{,}292$
fixed cells gives $\zl(N,N)\le1{,}223{,}138$.

For the lower bound, the graph of Lv--Lu--Fang~\cite[Theorem 4]{llf}
recalled in Section~\ref{sec:secondorder}, with $p=37$ and
$R=\{1,\ldots,18\}$, is tripartite and $C_4$-free with $666$ vertices per
part, degree $36$, and $35{,}964$ edges. Delete any $35$ vertices from each
part. At most $105\cdot36=3780$ edges are lost, so the remaining graph
has $631$ vertices per part and at least $32{,}184$ edges.
Use the same parts for
matrix rows and columns in the construction of
Theorem~\ref{thm:secondorder}. Its exact objective is
\[
 F+T=N^2/3+e(G)\ge1{,}194{,}483+32{,}184=1{,}226{,}667.
\]
This proves the finite separation. An explicit choice of the deleted
vertices, checked by exact integer computation in the ancillary witness
files, attains $1{,}226{,}771$; we retain the simpler bound. No uniqueness
of the maximizing plane incidence base or minimality of this dimension is
asserted.
\end{proof}

%% file: joint_bound.tex
\begin{lemma}\label{lem:lightmatrix}
Let a binary $m\times n$ matrix $A$ have $E$ ones partitioned into pairs.
For its row and column degrees $d_r,e_c$, put $\mu=E/(mn)$,
$u_r=d_r/n$, and $v_c=e_c/m$. Suppose $\epsilon\ge0$, $0\le\eta<1/100$,
every pair $((r,c),(s,d))$, allowing repeated lines, satisfies
\[
 u_r+u_s\le1+\epsilon\quad\text{or}\quad v_c+v_d\le1+\epsilon,
\]
and
\[
 S:=\frac1m\sum_r u_r^2+\frac1n\sum_c v_c^2\le\frac{4\mu}{3}+\eta.
\]
Then $E<(3/5)(1+\epsilon)mn$.
\end{lemma}
\begin{proof}
Suppose otherwise. Set $\theta=3/(5\mu)\le1$ and $W=\theta A$.
The weighted row and column degrees $U_r=\theta u_r$, $V_c=\theta v_c$
have common mean $3/5$, and $0\le W\le1$. Each pair has degree sum
at most one on one axis, since $\theta(1+\epsilon)\le1$.

Use the polynomials $a,b,h$ specified in Appendix~\ref{sec:certificate};
there $h$ is shown to be nondecreasing with $h(1-x)=-h(x)$.
The two endpoint minima of each pair
sum to at most one, and $x+y\le1$ gives
$h(x)+h(y)\le h(x)+h(1-x)=0$. As both cells have weight $\theta$,
\[
 \langle W_{rc}h(\min(U_r,V_c))\rangle\le0,
\]
where brackets denote the uniform average over all $mn$ positions.

Set $\lambda=-81/100$ and
\begin{align*}
 G(u,v)&=\lambda+u^2+v^2-ua(u)-va(v)
                    +(v-3/5)b(u)+(u-3/5)b(v),\\
 P(u,v)&=G(u,v)+a(u)+a(v)+h(u).
\end{align*}
Appendix~\ref{sec:certificate} proves $G,P\ge0$ on $[0,1]^2$.
Since $G$ is symmetric,
$G(u,v)+a(u)+a(v)+h(\min(u,v))\ge0$ as well.
Average its convex combination with $G$, with weights $W,1-W$.
For example, since $n^{-1}\sum_cW_{rc}=U_r$,
\[
 \langle Wa(U)\rangle=\frac1m\sum_r U_ra(U_r),\qquad
 \langle(V-3/5)b(U)\rangle
 =\left(\frac1n\sum_cV_c-\frac35\right)
   \frac1m\sum_r b(U_r)=0.
\]
Together with the transposed identities, these cancel the $a$ terms
and remove the $b$ terms. The averages here are over all grid positions;
no independence of row and column degrees on occupied cells is assumed.
Consequently
\[
 0\le\lambda+\theta^2S+
       \langle Wh(\min(U,V))\rangle
 \le-\frac{81}{100}+\frac45\theta+\theta^2\eta
 \le-\frac1{100}+\eta<0,
\]
a contradiction.

\end{proof}

For a matrix whose occupied cells are partitioned into (C3)-pairs,
\eqref{eq:degree-moment} and its transpose give the moment hypothesis with
\[
 \eta_{m,n}=\frac{2\mu}{3}\left(\frac1m+\frac1n\right)
       +\frac{m+n-2}{mn}=o(1)
\]
as both dimensions grow.

%% file: universal_lower.tex
\section{A sharp guarantee for every maximum base}\label{sec:everymaximum}
The quarter construction leaves certain fixed rows bare: they receive
no selected cells. These anchors ensure that each new pair has only one
possible (C3) witness, which a filter keeps empty. We first prove the
finite counting lemma, then choose enough row and column pairs to obtain
the quarter guarantee.

\begin{theorem}\label{thm:everymaximum}
As $\min\{m,n\}\to\infty$, every maximum $C_4$-free $m\times n$
fixed base admits an original augmentation with
\[
 T\ge(1/4-o(1))mn,
\]
uniformly over the dimensions and the base. For each fixed $c_0>0$,
the stronger bound $T\ge mn/4-O_{c_0}(m)$ holds uniformly when
$m\ge c_0n^2$. The constant $1/4$ is sharp:
when $n\to\infty$ and $m/n^2\to\infty$, there are maximum bases
whose best augmentation has objective $(1/4+o(1))mn$, whereas
$\zl(m,n)=(1/3+o(1))mn$.
More generally, the lower guarantee holds for arbitrary $C_4$-free bases $B$
along any sequence $m\ge n\to\infty$ with
\[
 z(m,n)-|B|=o\bigl(n\min\{n,\sqrt m\}\bigr).
\]
\end{theorem}

\begin{lemma}[Bare anchors]\label{lem:bareanchors}
Let $B$ be any $C_4$-free base with $F$ fixed cells. Choose disjoint
ordered row pairs $(P_i,Q_i)$, $1\le i\le R$, and disjoint ordered
column pairs $(y_j,d_j)$, $1\le j\le C$. Suppose each row pair has a
common fixed column $c_i$, and each column pair has a common fixed row
$A_j$ outside all paired rows. The rows $A_j$ may repeat. Put
\[
 h(x)=\#\{j:B_{A_j,x}=1\}.
\]
There is an admissible augmentation of the full base with
\[
 F+T\ge RC-\sum_{i=1}^R h(c_i).
\]
\end{lemma}
\begin{proof}
In the block $\{P_i,Q_i\}\times\{y_j,d_j\}$, add
$((P_i,d_j),(Q_i,y_j))$ if the block contains no fixed cell and
$B_{A_j,c_i}=0$. All other cells retain their fixed status. Disjoint
blocks give (S), and both opposite cells of each pair are empty, giving
(C2). A selected block is a checkerboard; an unselected block is bare.
Thus additions create no common neighbor of either paired set of lines.
By $C_4$-freeness their full common neighborhoods are exactly
\[
 N_O(y_j)\cap N_O(d_j)=\{A_j\},\qquad
 N_O(P_i)\cap N_O(Q_i)=\{c_i\}.
\]
The only possible (C3) witness is $(A_j,c_i)$. Its row is unpaired,
so this cell is empty by the filter. By disjointness, at most $F$ blocks contain fixed
cells, and exactly $\sum_i h(c_i)$ fail the filter before overlap with
that exclusion. Subtracting these from $RC$ proves the bound. The count
$h$ includes repeated anchor indices.
\end{proof}

\begin{lemma}\label{lem:universalquarter}
If $n\to\infty$ and $n\le m=o(n^2)$, every $C_4$-free $m\times n$
base with $F\sim z(m,n)$ admits $T\ge mn/4-o(mn)$.
\end{lemma}
\begin{proof}
Let $L$ count the uncovered column pairs. By
\eqref{eq:classicalasymptotic}, the actual fixed-cell count satisfies
$F\sim n\sqrt m$.
For the row degrees $b_r$ and $a=n/\sqrt m$, pair counting gives
\[
 0\le2L\le n(n-1)-F^2/m+F=o(n^2),\qquad
 V:=\sum_r(b_r-a)^2\le2n^2-n+F-2nF/\sqrt m=o(n^2).
\]
Hence $b_{\max}\le a+\sqrt V=o(n)$, and only $o(m)$ rows have
degree less than $a/2$, since their number is at most $4V/a^2$.

Choose a column $x_0$ minimizing its base degree plus its degree in
the leave graph. Their sum is at most $(F+2L)/n=o(n)$; write the
two degrees as $a_0,\ell_0$. Leave all $a_0$ rows through $x_0$
bare. Their fixed supports with $x_0$ removed are disjoint and cover
exactly $n-1-\ell_0$ columns. Pair as many columns as possible within each support, using
its row as the anchor. This gives
\[
 C\ge(n-1-\ell_0-a_0)/2=n/2-o(n).
\]
For any $x\ne x_0$, at most one anchor row contains $x$, and that row
supplies at most $b_{\max}/2$ column pairs. Thus $h(x)=o(n)$ uniformly
away from $x_0$, with repetitions counted.

Leave the rows of degree below $a/2$ bare too, and greedily pair
remaining rows whose fixed supports intersect. Unmatched supports
are disjoint and each has size at least $a/2$, so there are at most
$2n/a=2\sqrt m=o(m)$ unmatched rows. Therefore $R=m/2-o(m)$.
No paired row contains $x_0$, so every shared column $c_i$ differs
from $x_0$. Lemma~\ref{lem:bareanchors} now gives
$F+T\ge RC-Rb_{\max}/2=mn/4-o(mn)$.
Finally $F=o(mn)$. All estimates depend only on the dimensions, $F$,
and $C_4$-freeness, so the conclusion is uniform as $F/z(m,n)\to1$.
\end{proof}

\begin{lemma}\label{lem:profilequarter}
Fix $k\ge2$. In the range of Lemma~\ref{lem:classicalprofile}, every
$C_4$-free base, with deficit $\Delta=z(m,n)-F$, admits
\[
 F+T\ge mn/4-O_k(n^2+n\Delta).
\]
\end{lemma}
\begin{proof}
By Lemma~\ref{lem:classicalprofile}, the pair-count deficit
$D=M-kF+k(k+1)m/2$ is $O_k(n+\Delta)$.
Its leave-and-penalty identity implies the following estimates.
All but $O_k(n+\Delta)$ rows have degree $k$ or $k+1$, and the
exceptional rows have total squared degree $O_k(n+\Delta)$, since
$b^2\le(k+2)^2(b-k)(b-k-1)/2$ off those degrees. Color each column pair in
a regular row's support by that row. These cliques have order at most
$K=k+1$, are edge-disjoint, and leave $O_k(n+\Delta)$ pairs uncolored.

First fix a maximal pairing of regular rows with intersecting supports.
Unmatched supports are disjoint, so at most $n$ remain. There are
$R_0=m/2-O_k(n+\Delta)$ pairs. For their common columns $c_i$, put
$\lambda_x=\#\{i:c_i=x\}$; then $\sum_x\lambda_x=R_0\le m/2$.

On any $N=2\lfloor n/2\rfloor$ columns take a uniformly random perfect
matching. Let $U$ count its uncolored edges. For every original column $x$,
including the omitted column if $n$ is odd, let $G_x$ be the union of
regular-row cliques whose rows contain $x$, restricted to these $N$ vertices,
and let $Z_x$ count its matching edges. By $C_4$-freeness,
\[
 e(G_x)\le\frac K2\sum_{r\text{ regular}:\,B_{r,x}=1}(b_r-1)
 \le K(n-1)/2.
\]
Each edge belongs to the matching with probability $1/(N-1)$. Since the
weights $\lambda_x$ were fixed before this choice,
\[
 \mathbb E U=O_k(1+\Delta/n),\qquad
 \mathbb E\sum_x\lambda_xZ_x=O_k(m).
\]
Choose a matching for which
$U/(1+\Delta/n)+m^{-1}\sum_x\lambda_xZ_x=O_k(1)$.
Remove its uncolored edges, leaving
$C=n/2-O_k(1+\Delta/n)$ column pairs, and use their color rows as anchors.
The anchor of each pair is unique among all fixed rows by $C_4$-freeness.
Discard every previously chosen row pair touching an anchor row. At most
$n/2$ row pairs are discarded, so $R=m/2-O_k(n+\Delta)$ remain.

The indexed anchor count is $h(x)=Z_x$, so
$\sum_{i\text{ retained}}h(c_i)\le\sum_x\lambda_xZ_x=O_k(m)$.
Lemma~\ref{lem:bareanchors} gives $F+T\ge RC-O_k(m)$.
Since $R\le m/2$, $C\le n/2$ and $m\asymp_k n^2$, the product loss is at most
$(m/2)(n/2-C)+(n/2)(m/2-R)=O_k(n^2+n\Delta)$.
\end{proof}

\begin{proof}[Proof of Theorem~\ref{thm:everymaximum}]
We prove the stronger assertion. Write $\Delta=z(m,n)-F$ and suppose
$\Delta=o(n\min\{n,\sqrt m\})$ along a sequence $m\ge n\to\infty$.
First suppose $m\ge M$. Then $\Delta=m+M-F=o(n^2)$,
and \eqref{eq:terminaldeficit} applies.
Rows outside degrees one and two number at most $\Delta$ and have
total squared degree at most $9\Delta$. The degree-two rows therefore
cover $M-o(n^2)$ column pairs. A uniform matching on
$2\lfloor n/2\rfloor$ columns has $o(n)$ expected uncovered edges.
Choose one with at most this expectation.
Keep its covered edges and use their degree-two rows as bare anchors.
Then $C=n/2-o(n)$ and $h(x)\le1$ for every original column.
Leave the exceptional rows bare too, and greedily pair the other rows
with intersecting supports. At most $n$ nonempty supports remain
unmatched, so $R=m/2-o(m)$, using $m\ge M$.
Lemma~\ref{lem:bareanchors} gives $F+T\ge mn/4-o(mn)$.
When $\Delta=0$, all matching edges are covered:
$C=\lfloor n/2\rfloor$, $R=m/2-O(n)$, and $h\le1$.
Since $n^2+F=O(m)$, this gives $T\ge mn/4-O(m)$.

If the conclusion failed on a subsequence, the preceding case excludes $m\ge M$.
Pass to a subsequence with $m/n^2\to c\in[0,1/2]$. If $c=0$,
then $\Delta=o(n\sqrt m)$ and Lemma~\ref{lem:universalquarter} applies.
If $c>0$, then $\Delta=o(n^2)$ and only finitely
many classical-profile indices $k\ge2$ occur; pass to a fixed one
and use Lemma~\ref{lem:profilequarter}. Both contradict a fixed
positive density loss. Finally
$F/(mn)\le1/n+1/\sqrt m\to0$ uniformly, proving the claim for $T$.
Uniformity follows by choosing a failing sequence with $n\ge j$ and
$\Delta\le n\min\{n,\sqrt m\}/j$. The case $\Delta=0$ gives the
every-maximum conclusion, and transposition covers the other orientation.
The quantitative bound for $m\ge c_0n^2$ follows from the terminal case and,
below $M$, Lemma~\ref{lem:profilequarter} at $\Delta=0$ over finitely many
profile indices, using $n^2+F=O_{c_0}(m)$.

For sharpness, let $n\to\infty$ and $m/n^2\to\infty$. Take one
row for each column pair, and $S=m-M$ singleton rows all fixed at
column $1$. This base is $C_4$-free with $F=m+M$, hence maximum by
\eqref{eq:terminaldeficit}. At most $Mn=O(n^3)$ selected pairs meet
a pair row, by (S). Let $T'$ count the remaining pairs, and put
$E'=2T'$. These pairs use only the $S$ singleton rows and the $n-1$
nonanchor columns; let $e_c$ be their column degrees.
For each retained pair with columns $c,d$, at most two rows are
selected in both columns: any such row outside the pair's endpoint
rows would give a (C3) witness in fixed column $1$.
Thus $e_c+e_d\le S+2$, also when the columns repeat. Summing over
pairs and using Cauchy--Schwarz gives
\[
 \frac{(E')^2}{n-1}\le\sum_c e_c^2\le\frac{E'(S+2)}2.
\]
Consequently $T'\le(S+2)(n-1)/4$ (trivially if $E'=0$), and
$T^*(B)\le mn/4+O(n^3)=(1/4+o(1))mn$.
Together with the lower guarantee already proved, $F=o(mn)$, and
Theorem~\ref{thm:sharpasymptotic}, this proves sharpness and the stated
contrast with $\zl(m,n)=(1/3+o(1))mn$.
\end{proof}

%% file: universal_third.tex
\section{The one-third window for every maximum base}
The quarter guarantee is sharp in the worst case, but maximum bases in
the window $\binom n2\le m\le n^2$ all admit $mn/3-O(m)$ selected pairs.
We prove this by distributing singleton rows across three types and
matching cells between types. At each fixed limiting ratio above one,
a canonical base instead has a strict density deficit. The one-third
limit still holds when $m/n^2\to1$ from either side.

\begin{theorem}\label{thm:universalpairwindow}
Every maximum $C_4$-free $m\times n$ base $B$ with $m\ge M$ satisfies
\[
 T^*(B)\ge\frac n3\min\{m,n^2\}-20m.
\]
There is an absolute $\kappa>0$ such that, for $m\ge n^2$, the maximum
base $B_{\rm can}$ consisting of one row for every column pair and
$m-M$ singleton rows all fixed at column $1$ satisfies
\[
 T^*(B_{\rm can})\le\frac{mn}{3}-\kappa n(m-n^2)+4m.
\]
\end{theorem}

Thus $T^*(B)=mn/3+O(m)$ uniformly over all maximum bases on
$M\le m\le n^2$. More generally, every such base has limiting optimal
selected-pair density $1/3$ whenever $n\to\infty$, $m\ge M$, and
$m/n^2\to c\in[1/2,1]$,
including approach to $c=1$ from either side. For fixed $c>1$, the
canonical bases instead have upper limiting density at most
$1/3-\kappa(1-1/c)$. This is an obstruction for a specified maximum base.
Adding the $F=O(m)$ fixed cells does not change these density conclusions.

\begin{proof}[Proof of the lower bound]
By the zero-deficit case of \eqref{eq:terminaldeficit}, $B$ has one degree-two
$P$-row per column pair and $S=m-M$ singleton $S$-rows. For $n\le4$ the
claimed lower bound is negative, so assume $n\ge5$.

Let $g_c$ count the singleton rows at column $c$. Remove
$H=\{c:g_c>S/3\}$, which has size at most two, and the minimum number
of further columns needed so that the remaining set $K$ has size $N=3s$.
At most two further columns are needed. Write $R$ for
all removed columns and $r=|R|\le4$. Keep $R$ bare and leave every
$P$-row meeting $R$ bare. There are
$\ell=M-\binom N2\le rn$ such rows; all other $P$-rows have their
full fixed support in $K$.

Sort the singleton weights on $K$ decreasingly and assign them cyclically
to three classes $A,B,C$, each of size $s$. Their total weights
$W_A,W_B,W_C$ differ by at most the largest individual weight
$w\le S/3$: the extreme difference is a sum of disjoint segments of
the successive downward differences of the sorted list. Since their
sum is at most $S$, each $W_i\le S/3+2w/3\le5S/9$.

Give each differently colored $P$-row its type $AB,BC$, or $CA$.
There are $s^2$ of each. Put $q=\binom s2$. Split the $A$-internal
rows into $\lfloor q/2\rfloor$ of type $AB$ and the rest of type $CA$,
and cyclically for $B,C$. Every type has exactly
\[
 P_0=s^2+q=\binom N2/3
\]
$P$-rows. Choose balanced integer capacities $q_t$ for the singleton
rows, summing to $S$, each equal to $\lfloor S/3\rfloor$ or
$\lceil S/3\rceil$. A singleton anchored in an old color may use either
type containing that color; one anchored in $R$ may use any type.
An integral assignment filling these capacities exists by Hall's theorem:
for a single old color its available capacity is $S-q_{\rm opposite}$,
and
\[
 W_i\le\lfloor5S/9\rfloor\le\lfloor2S/3\rfloor
      =S-\lceil S/3\rceil\le S-q_{\rm opposite}.
\]
Any subset involving two old colors, or a removed-anchor row, has all
three types available. Thus the active row counts $R_t=P_0+q_t$ satisfy
$\min_tR_t\ge(m-\ell)/3-2/3\ge m/3-2n$.

A fiber consists of one old column on all rows of one of its two types.
It contains at least $P_0-n$ nonfixed $P$-cells. Put
\[
 h_0=\max\{0,\min\{\min_tR_t,\,2(P_0-n)\}\}.
\]
If $h_0=0$, select nothing. Otherwise, label the three color classes
by $i\in[s]$ and match the fibers
\[
 (AB,A_i)\leftrightarrow(BC,C_i),\quad
 (AB,B_i)\leftrightarrow(CA,C_i),\quad
 (CA,A_i)\leftrightarrow(BC,B_i).
\]
For each match take $h_i=\min\{h_0,a_i,a_i'\}$ cells from each side,
where $a_i,a_i'$ are its available nonfixed-cell counts. Each pool can
contain at least $\lceil h_i/2\rceil$ $P$-cells, since
$h_i\le2(P_0-n)$. Match the two pools avoiding $S$--$S$ pairs: their
singleton counts sum to at most $h_i$, so match each side's singleton
cells to the other side's $P$-cells, then match the remaining $P$-cells.

Every selected cell is used once, and both opposite cells lie outside
the endpoint row types, proving (S) and (C2). For a pair between
$(AB,A_i)$ and $(BC,C_i)$, common occupied old columns lie in $B$.
No common column lies in $R$, since one endpoint is a $P$-row with
no occupied cell in $R$. An active row occupying both $A_i,C_i$ must
have type $CA$ and therefore misses every old $B$-column. A bare
$P$-row has degree two and cannot occupy the three distinct columns
of a (C3) witness. The cyclic cases are identical. This verifies (C3)
in full occupancy, including singleton rows with fixed anchors in $R$.

Each fiber occurs once. Since $R_t\ge h_0$, the loss $h_0-h_i$ is at
most the sum of its two fixed-cell counts. Their total is at most
$2\binom N2+S\le m+M\le2m$. Thus, also when $h_0=0$,
\[
 T\ge Nh_0-2m.
\]
Writing $a=\min\{m,n^2\}/3$, the estimates above and $r\le4$ give
$h_0\ge a-5n$. Consequently
\[
 T\ge(n-r)(a-5n)-2m\ge na-\frac{19}{3}n^2-2m\ge na-20m,
\]
where $n^2\le5m/2$ for $n\ge5$, $m\ge M$. This proves the lower bound.
Theorem~\ref{thm:paireddensity}, after deleting fixed cells, supplies
$T\le mn/3+m/3+3n/4$, proving the stated one-third consequences.
\end{proof}

\begin{proof}[Proof of the canonical upper bound]
Fix any augmentation above $B_{\rm can}$, and put
\[
 u=\frac23-\frac{2T}{mn},\qquad
 \delta=u_++\frac1m+\frac1n,\qquad F=m+M\le3m/2.
\]
The parameter of Lemma~\ref{lem:cleanholes} is
$\eta=\delta+F/(mn)\le5\delta/2$. Fix a sufficiently small absolute
$\tau>0$. For sufficiently small $\delta$, that lemma gives partitions
$R_i,C_i$ with masses $r_i,p_i=1/3+O(\sqrt\delta)$, global full-occupancy
edit error $O(\delta mn)$, $O(\delta)$ bad-line fractions, and exact
good diagonal holes. We may assume $1/4\le r_i,p_i\le5/12$.
Discard for counting nonreciprocal pairs and pairs with a bad endpoint
row or column. By (S), their total number is $O(\delta mn)$.

Let $G_i$ be the good singleton rows in $R_i$, and $g_i=|G_i|$.
A good $P$-row with good fixed columns avoids $C_i$, while each bad
column meets exactly $n-1$ fixed $P$-rows. Thus
\[
 g_i\ge r_im-\tfrac12(1-p_i)^2n^2-O(\delta m)
      \ge m/3-2n^2/9-O(\sqrt\delta\,m)\ge m/10.
\]
The fixed anchor column $1$ need not be good. A retained reciprocal
pair of singleton endpoints in $R_i\times C_j$ and $R_j\times C_i$
would violate (C3): in the third class $G_k$, its two good selected
columns miss at most $2\tau m<m/10$ rows. A remaining row occupies
both columns and column $1$, as do the endpoint rows at column $1$.
The witness row is outside the endpoint classes, and the selected
columns differ from $1$. Hence no retained pair has two singleton endpoints.

Let $Q_{ij}=Q_{ji}$ count retained reciprocal pairs, and let $L_{ij}$
count their singleton endpoints in $G_i\times C_j$. The \emph{global}
edit bound, after removing $F=O(\delta mn)$ fixed cells and
$O(\delta mn)$ exceptional endpoints, gives
\[
 Q_{ij}=r_ip_jmn+O(\delta mn),\qquad
 L_{ij}=g_ip_jn+O(\delta mn)\quad(i\ne j).
\]
These aggregate errors use the global edit estimate, not the individual
$\tau$ good-line error. Symmetry of $Q_{ij}$ gives
$r_ip_j-r_jp_i=O(\delta)$; summing over $j$ yields
$r_i-p_i=O(\delta)$. Since $L_{ij}+L_{ji}\le Q_{ij}$, putting
$\sigma_i=g_i/(mp_i)$ gives
\[
 \sigma_i+\sigma_j\le r_i/p_i+O(\delta)\le1+O(\delta).
\]
For $\{i,j,k\}=\{1,2,3\}$, the weights
$\lambda_{ij}=(1-2p_k)/2$ are nonnegative and satisfy
$\lambda_{ij}+\lambda_{ik}=p_i$ and $\sum_{i<j}\lambda_{ij}=1/2$.
The weighted sum of these three inequalities therefore gives
\[
 \frac{\sum_i g_i}{m}=\sum_i p_i\sigma_i\le\frac12+O(\delta).
\]
Only $O(\delta m)$ singleton rows are bad. Since their total is $m-M$,
\[
 q:=\frac{m-n^2}{m}\le\frac{m-2M}{m}=O(\delta).
\]
For larger $\delta$ the bound $q\le K\delta$ still holds with an
absolute $K$, because $0\le q\le1$.

Finally put $\gamma=1/3-T/(mn)=u/2$ and $\zeta=u+4/m+4/n$.
Theorem~\ref{thm:paireddensity} gives
$u\ge-2/(3n)-3/(2m)$. Hence $\zeta\ge\delta$: this is immediate
for $u\ge0$, while for $u<0$,
$\zeta-\delta=u+3/m+3/n\ge3/(2m)+7/(3n)>0$.
Since $m\ge n^2$,
\[
 q\le K\delta\le K\zeta\le2K\gamma+8K/n.
\]
Rearranging proves the upper bound with $\kappa=1/(2K)>0$, including
when the finite signed deficit $\gamma$ is negative.
\end{proof}

%% file: threshold.tex
\section{The limited parameter's threshold}\label{sec:threshold}
The unrestricted density is always asymptotically $1/3$. For the limited
parameter, this density first becomes possible at $m/n^2=1/12$.
The next statement also measures the loss immediately below that threshold.

\begin{theorem}\label{thm:sharpasymptotic}
There are absolute constants $\varepsilon,c,C,n_0>0$ such that, for all
integers $n\ge n_0$ and $m\ge(1/12-\varepsilon)n^2$,
\[
 c n\left(\frac{n^2}{12}-m\right)_+-Cm
 \le\frac{mn}{3}-\zl(m,n)
 \le C\left[m+n\left(\frac{n^2}{12}-m\right)_+\right].
\]
Here $x_+=\max\{x,0\}$. In particular, for every fixed $D>0$,
$\zl(m,n)=mn/3+O_D(m)$ uniformly for sufficiently large $n$ and
$m\ge n^2/12-Dn$, and $\zl(\binom n2,n)=n^3/6+O(n^2)$.
\end{theorem}

\begin{theorem}\label{thm:threshold}
Fix $c>0$. If $c\ge1/12$, every integer sequence with
$n\to\infty$ and $m/n^2\to c$ satisfies
$\zl(m,n)/(mn)\to1/3$.
If $0<c<1/12$, there is $\varepsilon(c)>0$ such that every
sequence with $n\to\infty$ and $m/n^2\to c$ satisfies
\[
 \limsup\frac{\zl(m,n)}{mn}\le\frac13-\varepsilon(c).
\]
\end{theorem}

The obstruction is a capacity count: the exact diagonal holes from
Lemma~\ref{lem:cleanholes} leave too few column pairs for a maximum base
below the threshold. Matching constructions use rows of degrees three
and four near the boundary. Appendix~\ref{sec:constructions} gives the
capacity argument, the constructions, and the proofs of both statements.

%% file: open_problems.tex
\section{Concluding remarks and open problems}\label{sec:openproblems}
The unrestricted leading density is determined by the pairing conditions.
Requiring a maximum fixed base introduces a further obstruction, measured
by the separation below quadratic row growth and the threshold at
$m/n^2=1/12$. Several density questions remain.

\paragraph{The limited density below quadratic growth.}
For $n\le m=o(n^2)$, Theorem~\ref{thm:separation} leaves the interval
\[
 \frac14\le\liminf\frac{\zl(m,n)}{mn}
 \le\limsup\frac{\zl(m,n)}{mn}\le\frac3{10}.
\]
Does the normalized limited parameter converge even in the square case
$m=n$? More generally, is there a common limiting constant throughout
this regime, or does the answer depend on the growth of $m$ relative
to $n$? The sharp quarter guarantee for a supplied maximum base does
not settle this question, since $\zl$ optimizes over all maximum bases.

\paragraph{The size of the separation.}
The unrestricted leading term and the preceding bounds give
\[
 \frac1{30}\le
 \liminf\frac{\za(m,n)-\zl(m,n)}{mn}
 \le\limsup\frac{\za(m,n)-\zl(m,n)}{mn}
 \le\frac1{12}
\]
along every sequence in the same regime. Determining the best uniform
lower coefficient, or the limiting gap for squares, would sharpen the
answer to the original equality question. The exact certificate proves
the coefficient $1/30$ as a lower guarantee; it does not establish its
optimality.

\paragraph{The density profile below the threshold.}
For fixed $0<c<1/12$, does $\zl(m,n)/(mn)$ have a limit depending only
on $c$ whenever $m/n^2\to c$? Theorem~\ref{thm:threshold} separates it
from $1/3$, while Theorem~\ref{thm:sharpasymptotic} bounds the loss on
both sides near the threshold. An exact profile would also determine
the leading coefficient of that loss as $c\uparrow1/12$.

The pivot argument uses disjoint open-rectangle neighborhoods, and the
stability argument recovers the structure forced by near equality.
It is natural to ask which related augmentation conditions admit an
analogous description. For the motivating sum-of-squares problem, a
further question is whether such structure can guide constructions
beyond the original admissibility conditions, where the relaxed bounds
discussed in the introduction are already stronger in one family.

%% file: affine_core.tex
\section{Maximum-base constructions for the critical threshold}\label{sec:constructions}
This appendix proves Theorems~\ref{thm:sharpasymptotic} and~\ref{thm:threshold}.
We first prove the obstruction below the threshold, then construct
maximum bases at and above it. The final subsection assembles the bounds.

\subsection{The subcritical obstruction}
\begin{lemma}[Classical degree profile]\label{lem:classicalprofile}
Fix an integer $k\ge2$ and put $M=\binom n2$. Uniformly for sufficiently large $n$
and $M/\binom{k+1}2\le m\le M/\binom k2$,
\[
 z(m,n)=\frac Mk+\frac{k+1}{2}m+O_k(n).
\]
For every extremal base its deficit
$D=M-kz(m,n)+k(k+1)m/2$ is nonnegative and $O_k(n)$.
\end{lemma}
\begin{proof}
The identity
$\binom b2-kb+k(k+1)/2=(b-k)(b-k-1)/2\ge0$ gives the upper bound.
For $k\ge4$, use \cite[Theorem~1.1]{designexistence} to choose a fixed sufficiently large
$v_0\equiv1\pmod{k(k^2-1)}$ admitting both $2$-$(v_0,k,1)$ and
$2$-$(v_0,k+1,1)$ designs. For $k=2$, take $v_0=7$: all pairs and
the translates of $\{0,1,3\}$ modulo $7$ give the two designs.
For $k=3$, take $v_0=13$: the translates of $\{0,1,4\}$ and
$\{0,2,7\}$ give the triples, and those of $\{0,1,3,9\}$ give the
quads. For each cyclic design, the listed starter blocks have every
nonzero ordered difference exactly once. The same theorem gives a $2$-$(n',v_0,1)$
design for
\[
 n'=1+v_0(v_0-1)\left\lfloor\frac{n-1}{v_0(v_0-1)}\right\rfloor.
\]
Its blocks cover $P=\binom{n'}2=M-O_k(n)$ pairs. Fill each block with
either internal design. The attainable row counts run from
$P/\binom{k+1}2$ to $P/\binom k2$ in a fixed integer step. Choose the
largest count $S\le m$; then $m-S=O_k(n)$, including at the endpoints.
Append $m-S$ singleton rows. Pair-disjointness makes the base $C_4$-free,
and its fixed-cell count is
\[
 \frac Pk+\frac{k+1}{2}S+(m-S)
 =\frac Mk+\frac{k+1}{2}m
   -\frac{M-P}{k}-\frac{k-1}{2}(m-S).
\]
This proves the lower bound. Horsley and Pike's theorem is used only
for ordinary designs: its color-two, index-one case applies to every
fixed block size at least four. The explicit small seeds need no
additional congruence; the displayed choice of $n'$ ensures macro-design
admissibility in every case.
\end{proof}

The obstruction at a fixed positive quadratic ratio below $1/12$ is a
capacity count. A maximum base with $m\le M/6$ has almost all row degrees
in $\{k,k+1\}$ for one of finitely many fixed $k\ge4$, and by the clean holes each good row has its fixed support inside
two column classes. A row of degree $b$ requires at least $\frac49\binom b2-\frac23$
internal pairs, while the three classes together contain only about
$M/3$ internal pairs. Summing and comparing the two counts forces
$m\ge n^2/12-O(\delta n^2+n)$, where $\delta$ measures the density defect.

\begin{lemma}[Compact subcritical gap]\label{lem:linearcriticalgap}
\label{lem:compactsubcriticalgap}
For each fixed $0<c_0<1/12$ there are $\kappa,C,N>0$, depending only on
$c_0$, such that, for $n\ge N$ and $c_0n^2\le m\le n^2/12$,
\[
 \frac{\zl(m,n)}{mn}
 \le\frac13-\kappa\left(\frac1{12}-\frac m{n^2}\right)+\frac Cn.
\]
In particular the constants are absolute on $n^2/16\le m\le n^2/12$.
\end{lemma}
\begin{proof}
Take an extremizing configuration, put $M=\binom n2$, $F=z(m,n)$, and
write $\alpha=2T/(mn)$ and $\delta=(2/3-\alpha)_++1/m+1/n$.
The bound $F\le m+M$ gives
$F/(mn)\le(1+1/(2c_0))/n$. If $\delta$ is sufficiently small in terms
of $c_0$, apply Lemma~\ref{lem:cleanholes} with a fixed small $\tau$;
its parameter is $O_{c_0}(\delta)$. Thus bad-line fractions are
$O_{c_0}(\delta)$, good diagonal holes are exact, and squared column
class-mass deviations are $O_{c_0}(\delta)$.

If $m>M/6$, then $1/12-m/n^2\le1/(12n)$, which will be absorbed below.
Otherwise choose $k\ge4$ from the actual interval of
Lemma~\ref{lem:classicalprofile}. Since $k(k-1)\le1/c_0$, only finitely
many $k$ occur and their constants and thresholds are uniform. Exact
base extremality gives
\[
 D=\left(M-\sum_r\binom{b_r}{2}\right)
    +\sum_r\frac{(b_r-k)(b_r-k-1)}2=O_{c_0}(n).
\]
The leave has $O_{c_0}(n)$ pairs. Rows outside degrees $k,k+1$ have
total squared degree $O_{c_0}(n)$, using
$b^2\le(k+2)^2(b-k)(b-k-1)/2$ off those degrees; discard them.
Every remaining row has bounded degree in terms of $c_0$, and each
column meets at most $(n-1)/(k-1)$ such rows. Discard bad rows and
rows through bad columns. The additional pair loss is
$O_{c_0}(\delta n^2)$, so the retained rows cover
$M-O_{c_0}(\delta n^2+n)$ distinct fixed pairs. Their supports lie
in two column classes by the exact holes.

For every integer $b\ge0$,
\[
 I_b=\binom b2-\left\lfloor b^2/4\right\rfloor
 \ge\frac49\binom b2-\frac23.
\]
The difference is $(b-4)(b-6)/36$ for even $b$, and
$((b-5)^2+8)/36$ for odd $b$, both nonnegative.
The retained rows require at least
$4M/9-2m/3-O_{c_0}(\delta n^2+n)$ internal pairs. Writing
$p_i=|C_i|/n$, their capacity is exactly
\[
 \sum_i\binom{|C_i|}{2}
 =n^2/6-n/2+\frac{n^2}{2}\sum_i(p_i-1/3)^2
 \le M/3+O_{c_0}(\delta n^2).
\]
Comparison yields
\begin{equation}\label{eq:criticaldeficitdelta}
 1/12-m/n^2\le C_1\delta+C_2/n,
\end{equation}
with constants depending only on $c_0$. For larger $\delta$ this follows
by enlarging $C_1$, because the left side is at most $1/12-c_0$.
The case $m>M/6$ is absorbed by $C_2/n$.

Put $\gamma=1/3-\zl(m,n)/(mn)$ and
$\zeta=2/3-\alpha+4/m+4/n$. The paired bound gives
$\alpha\le2/3+2/m+2/n$, hence $\zeta\ge\delta$. For large $n$,
$m\ge n$, and the fixed-cell bound gives
\[
 \zeta\le2\gamma+(10+1/c_0)/n.
\]
Substitute into \eqref{eq:criticaldeficitdelta} and rearrange. This also
handles finite augmented densities above $1/3$ without truncating
$\gamma$ to its positive part.
\end{proof}

\begin{theorem}[Subcritical gap]\label{thm:subcriticalgap}
For every $0<c<1/12$, there are $\epsilon,\eta>0$ and $N$ such that
$\zl(m,n)/(mn)\le1/3-\epsilon$ whenever $n\ge N$, $m\ge n$ and
$|m/n^2-c|<\eta$.
\end{theorem}
\begin{proof}
In Lemma~\ref{lem:compactsubcriticalgap}, take $c_0=c/2$,
$\eta<\min\{c/2,(1/12-c)/2\}$ and $\epsilon=\kappa(1/12-c)/4$,
then choose $N$ so that $C/n\le\epsilon$.
\end{proof}

We call a four-point block a quad. A bare row receives no selected cells;
a row of type $AB$ has all its fixed and selected cells in $A\cup B$,
and the other two-color types are defined similarly.

\subsection{A local trade}
\begin{lemma}\label{lem:affinetradeprogression}
On three colored copies of five indices there are fifteen quads,
five of each type $AB,BC,CA$, covering every pair of differently indexed
columns exactly once. Each five-quad group may be replaced by ten triples
of the same type and with the same covered pairs. Within that group,
each point's incidence changes from two to three.
\end{lemma}
\begin{proof}
On two five-sets take the quads
\[
 \{A_j,A_{j+1},B_{j-1},B_{j+2}\},\qquad j\in\mathbb F_5.
\]
Under the identification
$A_i\leftrightarrow\{i-1,i\}$ and $B_i\leftrightarrow\{i-2,i+1\}$,
these are the five stars of $K_5$ on its ten edge-points.
Replace them by the ten triples arising from triangles of $K_5$:
each pair of incident edges belongs to one star and one triangle.
Both configurations cover the same thirty pairs once.
The two color classes are complementary pentagons, so every new triple
has two points of one color and one of the other.
Each edge-point belongs to two stars and three triangles.
These uncolored configurations are classical~\cite{brouwer}; the labeling
specifies their use here.

Use the same local labeling for the cyclic groups $AB,BC,CA$.
Within a color, one group covers index differences $\pm1$ and the other
covers $\pm2$. Between two colors, precisely the unequal-index pairs
are covered. Their union therefore covers exactly the required ninety
pairs, without repetition.
\end{proof}

%% file: construction_frame.tex
\subsection{A common construction on groups of indices}\label{sec:frame}
We use the following consequence of the dense decomposition theorems
of Barber, K\"uhn, Lo, and Osthus~\cite{densedecomp}.
For each $k\in\{3,4,5\}$ there are constants $0<\delta_k<1/2$ and $v_k$
such that a graph on $v\ge v_k$ vertices has a $K_k$ decomposition
if its minimum degree is at least $(1-\delta_k)v$, its degrees are
divisible by $k-1$, and its edge count is divisible by $\binom k2$.
We also use \cite[Lemma 3.5]{chm-unbalanced}: nonnegative prescribed
point degrees $d_x$ are realized by a linear $k$-uniform hypergraph if
\begin{equation}\label{eq:prescribeddegrees}
 \sum_xd_x=kb>0,\qquad R=\max_xd_x,
 \qquad b\ge2R\bigl(1+(k-1)R\bigr).
\end{equation}
Here linear means that no point pair occurs in two hyperedges.

Write $n=3s+t$, where $s=20\lfloor n/60\rfloor$ and $0\le t<60$.
Take $s\ge v_5$ and partition the $s$ indices into groups of sizes $g_i$, all divisible by
$20$, with $\max_i g_i\le\delta_5s$. One group is central. Each old index has
three colored copies; the remaining $t$ columns are new.
The complete multipartite graph on the indices is $K_5$-divisible and
has minimum degree $s-\max_i g_i$. Decompose it into
\[
 B=\frac{s^2-\sum_i g_i^2}{20}
\]
copies of $K_5$. On each use the three cyclic local gadgets of
Lemma~\ref{lem:affinetradeprogression}, with a shared labeling.
The resulting $15B$ typed quads cover exactly the $90B$ colored column
pairs whose indices belong to different groups. Each of the $3B$
five-star groups can independently be traded for ten triples, changing
the row and fixed-cell counts by $(5,10)$.
Each typed triple may also be split into its three pair rows, preserving
its pair set and type while changing those counts by $(2,3)$.

An ordinary group with $v_i=3g_i$ old columns is allotted the pair set
on its old columns and the $t$ new columns, excluding new--new pairs;
its size is
\[
 P_i=\binom{v_i}{2}+v_it\equiv0\pmod6.
\]
The central group, together with the new columns, has order
$v_c=3g_c+t$ and receives its entire pair set. These local budgets and
the bulk budget partition all column pairs. Fill each local budget with
a $C_4$-free base using only its allotted pairs, and keep all local rows
bare. The global base is then $C_4$-free. In particular, the center may
have a leave or arbitrary row degrees.
Typed rows have fixed cells only in their two named colors of old columns;
each bare row has old columns from a single index group.

\begin{lemma}\label{lem:frameaugmentation}
Suppose the assembled base is maximum, has $m\ge n^2/24$ rows,
$L$ bare rows, and typed row counts differing by at most five.
Then it admits an augmentation with
\[
 T\ge\frac{mn}{3}-O(m+nL).
\]
Consequently $L=O(n)$ gives $\zl(m,n)=mn/3+O(m)$, and
$L=o(m)$ gives $\zl(m,n)=mn/3+o(mn)$ as $n\to\infty$.
\end{lemma}
\begin{proof}
Let $R$ be the smallest typed row count, and let $F_{\rm typed}$ count
the fixed cells in typed rows. A typed fiber consists of the cells in
one old column over all rows of its type.
Order the index groups consecutively and shift cyclically by their
maximum size. Once $s\ge2\max_i g_i$, the resulting permutation
$\sigma$ sends every index outside its group.
Pair reciprocal fibers as follows:
\[
 (AB,A_i)\leftrightarrow(BC,C_{\sigma(i)}),\quad
 (AB,B_i)\leftrightarrow(CA,C_{\sigma(i)}),\quad
 (CA,A_i)\leftrightarrow(BC,B_{\sigma(i)}).
\]
For each of these $3s$ fiber pairs, choose as many nonfixed cells as
possible on each side, with equal sizes $h_q$, and match them by any
bijection. If $f_q,f'_q$ are the two fixed-cell counts, then
$h_q\ge R-f_q-f'_q$. Each typed fiber occurs exactly once, so
\[
 T=\sum_q h_q\ge3sR-F_{\rm typed}
    \ge s(m-L)-10s-F_{\rm typed}.
\]
Zero-sized matchings are allowed. These matchings use each selected cell once and have both opposite cells
empty. The common occupied columns of two endpoint rows lie in their
shared color. A typed row containing both selected columns uses the other
two colors, so cannot complete a (C3) witness. A bare row has old columns
from only one index group, so cannot contain both selected columns.
For example, an $AB$-row selected at an $A$-column is paired with a
$BC$-row selected at a $C$-column. Their common columns lie in $B$,
whereas a typed row occupying both selected columns has type $CA$ and
misses $B$. The index shift excludes a bare row from occupying both
selected columns, and the other two matches follow cyclically.
No new column is selected or occupied in an endpoint row. This proves
literal (C3) in full occupancy.
Since $n=3s+t$, $t<60$, and $F_{\rm typed}\le F\le m+M\le13m$,
the count gives $T\ge mn/3-O(m+nL)$.
Finally \eqref{eq:generaloriginalupper} and $F=O(m)$ give the
asserted matching upper estimates.
\end{proof}

In the bounded-group applications, take ordinary index size $g=20w$
and central size $g_c\in[g,2g)$, with $w$ sufficiently large in terms
of the fixed parameters of that application. Write
\[
 v=3g=60w,\quad v_c=3g_c+t,\quad
 P=\binom v2+vt,\quad M_c=\binom{v_c}{2},\quad qv+v_c=n,
\]
where $q$ is the number of ordinary groups.

%% file: critical_boundary.tex
\subsection{The full critical boundary}\label{sec:criticalboundary}
Put $\rho=(n-1)\bmod3$ and $\beta=(1-n)\bmod3$.
The relevant classical bounds are
\[
 A_4=\frac{M+10m}{4},\qquad A_3=\frac M3+2m,
 \qquad
 B_\rho=\Phi_\rho(M,m,n),
\]
where the last expression is used only for $\rho=1,2$, and
\[
 \Phi_\rho(P,r,x)=\frac{a_\rho P+90r-2x}{b_\rho},
 \qquad (a_1,b_1)=(13,42),\quad(a_2,b_2)=(11,39).
\]
Write $\ell_n=M/6-2\rho n/15$ and $u_n=M/6+\beta n/6$.
Chen--Horsley--Mammoliti~\cite[Theorem 1.3]{chm-unbalanced} give the
floors of $A_4,B_\rho,A_3$, respectively, on
\begin{equation}\label{eq:criticalknownranges}
 m\le\ell_n-52,\qquad
 \ell_n+24\le m\le u_n-120,\qquad
 m\ge u_n+40
\end{equation}
within any fixed linear window about $M/6$, for sufficiently large $n$.
The middle range is omitted when $\rho=0$.
For the first and third ranges the other endpoints in the cited theorem
are $(M+((1-n)\bmod4)n)/10+68$ and
$M/3-((n-1)\bmod2)n/4-28$, respectively; both are outside such a window.
These are known classical values. We first arrange compatible maxima
on their ranges, then treat the remaining strips without assuming any
formula for their classical maxima.

\begin{lemma}\label{lem:criticalknown}
Fix $D>0$. In the ranges \eqref{eq:criticalknownranges} with
$|m-M/6|\le Dn$, one has $\zl(m,n)=mn/3+O_D(m)$.
\end{lemma}
\begin{proof}
Use the bounded-group frame, with $w$ sufficiently large in terms of $D$.
Keep the bulk untraded and all local rows bare.

For the $A_3$ range, prescribe a linear triple family on the old ordinary
points with
\[
 T_3=\beta v/3+2j,\qquad 0\le j\le J:=\lceil2(D+1)v\rceil,
\]
and point degrees $\beta+3k_x$, where $\sum_x k_x=2j$ and the
nonnegative $k_x$ differ by at most one. For $A_4$ instead use
\[
 T_5=\rho v/5+3j,\qquad \sum_x k_x=5j,
\]
five-rows with point degrees $\rho+3k_x$.
Criterion~\eqref{eq:prescribeddegrees} supplies these families for large
$v$, with just the following omissions: when $\beta=0$, omit
$1\le j\le20$; when $\rho=0$, omit $1\le j\le25$.
The choice $j=0$ then means the empty family.
Indeed, in the zero-residue cases with average increment at most one,
the maximum point degree is three and the criterion requires respectively
$T_3\ge42$ and $T_5\ge78$. Above that range, or with positive residue,
the block count grows linearly in $v$ while its maximum degree is
$O_D(1)$, so the criterion holds uniformly.

Remove the prescribed pairs and $K_t$ from $K_{v+t}$.
Old residual degrees are divisible by three because
$2\beta\equiv\rho\equiv t-1\pmod3$; new degrees are $v$.
The edge count is divisible by six: $T_3$ is even and $T_5$ is a multiple
of three. The missing degrees are bounded, so dense $K_4$ decomposition
completes the fillers. Their row and fixed-cell counts satisfy
\[
\begin{array}{c|c|c}
 &m_o-P/6&F_o\\ \hline
 A_3&\beta v/6+j&P/3+2m_o\\
 A_4&-2\rho v/15-2j&(P+10m_o)/4
\end{array}
\]
and the displayed fixed-cell expressions are integers.

For $A_3$, use the 21 consecutive central row counts starting at
$c_3=\lceil M_c/6+\beta v_c/6+40\rceil$.
For $A_4$, use the 52 consecutive counts ending at
$c_4=\lfloor M_c/6-2\rho v_c/15-52\rfloor$.
They lie in the actual classical intervals for large fixed $v_c$.
For $J\ge51$, sums of $q$ members of
$\{0\}\cup\{h,\ldots,J\}$ give
$\{0\}\cup\{h,\ldots,qJ\}$ for $h=21$ or $26$.
The central intervals bridge the first omitted counts and the step-two
parity in $A_4$. Since $(M-M_c)/6=15B+qP/6$ is integral, this gives
every integer from $\lceil u_n+40\rceil$ through that value plus $qJ$,
and from $\lfloor\ell_n-52\rfloor-2qJ$ through
$\lfloor\ell_n-52\rfloor$, respectively.
These cover the requested side ranges because $qv\ge n/2$ eventually.
Adding each central maximum to the integer outside cap expression gives
exactly the global cap floor.

For the middle range the ordinary nonquad families are as follows:
\[
\begin{array}{c|c|c|c|c|c}
 \rho&T_3&T_5&j&m_o-P/6&F_o-2P/3\\ \hline
 2&20w-10j&12j&0\le j\le2w&10w-13j&20w-30j\\
 1&40w-10j&3j&0\le j\le4w&20w-7j&40w-15j
\end{array}
\]
For $\rho=2$, take disjoint triples and the dual of a simple
5-regular graph on $T_5$ vertices, on disjoint supports.
For $\rho=1$, use the dual of a simple cubic graph on $T_3$ vertices
and disjoint five-sets. A graph dual has its edges as points and the
edges incident to a vertex as one block; simplicity makes it linear.
The required regular graphs exist whenever their order is positive:
a cycle and its opposite matching suffice for degree three; for degree
five add the edges of differences $\pm2$ to that construction.
Their positive orders here are at least ten and twelve, respectively.
The supports use exactly $v$ points. Removed pair degrees are $2$ or $8$
when $\rho=2$, and $4$ when $\rho=1$. The pair counts are divisible by
six. Thus $K_4$ completion gives the table.

In both rows $F_o=\Phi_\rho(P,m_o,v)$ is integral, and the offsets run
from $-2\rho v/15$ to $\beta v/6$ in steps of 13 or 7.
Use the full central interval
\[
 [\lceil M_c/6-2\rho v_c/15+24\rceil,
   \lfloor M_c/6+\beta v_c/6-120\rfloor].
\]
The sum of the $q$ ordinary progressions has the same step.
The central interval exceeds that step for large $w$, so adding it to the ordinary
progression covers exactly the middle range in
\eqref{eq:criticalknownranges}.
The outside fixed count is the integer
$\Phi_\rho(M-M_c,m-m_c,qv)$; adding the central floor gives
$\lfloor\Phi_\rho(M,m,n)\rfloor$ because $qv+v_c=n$.
The frame ensures pair-disjointness, so all these bases are maximum.
There are $O_D(n)$ bare rows and equal typed row counts.
Lemma~\ref{lem:frameaugmentation} completes the proof.
\end{proof}

\begin{lemma}\label{lem:criticalstrips}
Uniformly for sufficiently large $n$, one has $\zl(m,n)=mn/3+O(m)$
on the remaining strips
\[
 [\ell_n-52,\ell_n+24],\quad[u_n-120,u_n+40]
 \quad(\rho=1,2),
\]
and $[M/6-52,M/6+40]$ when $\rho=0$.
\end{lemma}
\begin{proof}
Start with any actual maximum base, with $F=z(m,n)$ cells.
We will preserve its row-size counts and leave while rearranging its
nonquad rows into components of bounded order. The common frame can then
rebuild all quads around those components, giving another maximum base
whose bulk rows admit the augmentation. This avoids needing a formula
for $z(m,n)$ inside the strips.

The largest integer not exceeding the left endpoint of each strip has
a known classical maximum by
\eqref{eq:criticalknownranges}. Appending $O(1)$ singleton rows gives
a base at the target row count within $O(1)$ of both neighboring real
caps. Thus the maximum itself has bounded deficits in those two caps.
Indeed, the neighboring caps agree at the switch, and their slopes and
the strip widths are fixed. All constants here are therefore absolute.

Let $\mathcal L$ be its leave graph of uncovered point pairs, with $e$
edges, and let $b_r$ be its row sizes. The exact identities are
\begin{equation}\label{eq:criticalconvexdeficits}
 M-4F+10m=e+\sum_r\frac{(b_r-4)(b_r-5)}2,
 \quad
 M-3F+6m=e+\sum_r\frac{(b_r-3)(b_r-4)}2.
\end{equation}
At a lower switch the first is bounded; at an upper switch the second
is bounded. Consequently the leave, the number of rows outside
$\{4,5\}$ or $\{3,4\}$, and their total incidences are bounded.
This includes empty and singleton rows.

Suppose first that $\rho\in\{1,2\}$. For the modular cap put $j=1$ if
$\rho=2$ and $j=2$ if $\rho=1$, and
write $[z]_3$ for the least nonnegative residue of $z$ modulo three.
At each point define
\[
 S_j(x)=j\deg_{\mathcal L}(x)
          +\sum_{r\ni x}[j(b_r-1)]_3\equiv2\pmod3.
\]
Thus $S_j(x)\ge2$. With
\[
 \delta_j(b)=(9+2j)\binom b2+90-(36+3j)b-b[j(b-1)]_3,
\]
the bounded modular deficit has the exact expansion
\[
 (9+2j)M+90m-(36+3j)F-2n
 =9e+\sum_r\delta_j(b_r)+\sum_x(S_j(x)-2).
\]
Each $\delta_j(b)$ is nonnegative. Check $0\le b\le5$ directly;
for $b\ge6$ use
$2\delta_j(b)\ge(9+2j)b^2-(85+8j)b+180>0$.
Hence only boundedly many points have $S_j(x)>2$.

Mark the set $X$ of all such points, all exceptional nonquad-row
supports, and all leave endpoints. It is bounded. Put $k=5$ at a lower
switch and $k=3$ at an upper switch; thus exceptional rows have sizes
outside $\{4,k\}$. The coefficient of a $k$-row in
$S_j$ is one or two, whereas a quad has coefficient zero. Every $S_j(x)$
is bounded because $\sum_x(S_j(x)-2)$ is bounded and nonnegative termwise.
Summing $S_j(x)$ over the bounded set $X$ therefore bounds the number of
$k$-row incidences at $X$, and hence the number of $k$-rows meeting $X$.
Outside $X$ there is no leave edge or exceptional row, and $S_j(x)=2$.
Every such point consequently has exactly
$\nu$ incidences in $k$-rows, where
\[
 (\rho,\text{switch},k,\nu)
 =(2,\text{lower},5,2),\ (1,\text{lower},5,1),\
 (2,\text{upper},3,1),\ (1,\text{upper},3,2).
\]
All original quads will be rebuilt; their incidences at $X$ need not be
bounded.

If $\nu=1$, the $k$-rows meeting $X$ have private points outside $X$.
Together with the exceptional rows and leave they form a bounded core;
all other nonquad components are disjoint $k$-sets.
If $\nu=2$, encode each outside point as an edge joining its two
$k$-rows. There are no loops, and two row-vertices cannot have parallel
edges because the original base is $C_4$-free. This is therefore a simple
graph. Mark the boundedly many row-vertices
meeting $X$, recording their subsets $X_r\subseteq X$.
Their graph degrees are $k-|X_r|$; other vertices have degree $k$.
Forbid an edge between marked rows already sharing a point of $X$.
Such an edge would give those rows a second common point. Conversely,
simplicity and these forbidden adjacencies ensure that no two reconstructed
$k$-rows share two points. They also prevent every new collision with
the retained structure: exceptional rows and leave edges lie on $X$,
where the subsets $X_r$ remain unchanged.

Consider these marked profiles up to relabeling of $X$ and the marked
row-vertices. A profile records all subsets $X_r$, the exceptional-row
multiset, and the leave on $X$. The bounds already proved on their sizes
and multiplicities show that only finitely many profiles occur; empty
exceptional rows are retained in the multiset as well.
For each feasible profile and residue of the total number of $k$-rows
modulo $k+1$, choose a graph satisfying these restrictions of minimum
order $R_0$. If the original graph has order $R$, it is a candidate for
that same profile and residue. Thus $R_0\le R$ and $R-R_0$ is divisible
by $k+1$. Replace the original graph
by the chosen graph and $(R-R_0)/(k+1)$ disjoint copies of $K_{k+1}$.
Each added component has degree $k$ and no marked vertices, so all degree
and forbidden-adjacency requirements continue to hold. The marked data
and the total number $R$ of $k$-rows are unchanged. So is the outside
point count: points correspond to graph edges, and the handshake identity
gives
\[
 |V\setminus X|=\bigl(kR-\sum_{r\text{ marked}}|X_r|\bigr)/2.
\]
The maximum of the finitely many minimum orders $R_0$ is an absolute
constant. The chosen graph, $X$, the exceptional rows, and the leave
therefore form a bounded core. Each added complete graph becomes a
disjoint normal component under the inverse encoding. We have preserved
every nonquad row-size count and the leave, up to the common relabeling
of points. At points of $X$, all retained row incidences are unchanged;
at outside points, there are still exactly $\nu$ incident $k$-rows.
When $\rho=0$, adding \eqref{eq:criticalconvexdeficits} gives
$2e+\sum_r(b_r-4)^2=O(1)$; all nonquad rows and leave already have
bounded support, so they form the core without a graph replacement.
In this case take $X$ to be that support; the other points are isolated
in the retained structure.

The possible normal components have point count $p$ and covered-pair
count $e_0$ as follows:
\[
\begin{array}{c|ccccc}
 \text{component}&\text{isolated point}&\text{triple}&\text{five-set}
                    &\text{dual }K_4&\text{dual }K_6\\ \hline
 p&1&3&5&6&15\\
 e_0&0&3&10&12&60
\end{array}
\]
Each case uses only one of these component types outside its core.
Choose the constant frame size $v=60w$ sufficiently large for all cores.
Put $v/p$ normal components in each ordinary old group and the core
plus the remaining components centrally. This uses exactly $n$ points:
if the core has $b$ points, then $n-b$ and $v_c-b$ are multiples of $p$,
since $p\mid60$ and $n-v_c=qv$. Increase $w$ to ensure $v_c\ge b$.
Thus there are enough complete normal components to fill the ordinary
groups, and the remainder fits the center exactly. Relabel the points
accordingly; all $t$ new columns belong to the center. The leave and every
nonquad row are now contained in one local budget of the frame.

Let $H$ be the pair graph of the retained nonquad rows together with
the leave. Linearity makes these edge sets disjoint. It has bounded
maximum degree, and preserves the original
conditions
\[
 n-1-\deg_H(x)\equiv0\pmod3,
 \qquad M-|E(H)|=6Q,
\]
where $Q$ was the number of original quad rows.
For the degree congruence, the retained pair degree at a point of $X$
is unchanged, and an outside point has retained pair degree
$\nu(k-1)$; when $\rho=0$, outside points have degree zero.
For the edge count, preserving the leave and every nonquad row size
preserves $|E(H)|$.
In an ordinary group the graph $K_{v+t}-K_t-H_i$ has degrees divisible
by three and edge count $P-(v/p)e_0$ divisible by six.
Indeed, old residual degrees are congruent to the corresponding global
residual degrees, new degrees are $v$, and both $P$ and $(v/p)e_0$ are
multiples of six for every component in the table.
Its missing degrees are bounded. Centrally the degree congruence follows
from $v_c\equiv n\pmod{60}$, and edge divisibility follows by subtracting
the ordinary residual budgets and the $90B$ bulk pairs from $6Q$.
The core alone need not be divisible: it is the central residual graph
that satisfies the decomposition conditions. Choose $w$ large enough
for the fixed missing-degree bounds and all decomposition thresholds.
Dense $K_4$ decomposition then completes both ordinary and central graphs.

The resulting base has exactly the original leave and nonquad row counts.
Together with the bulk quads, these decompositions cover exactly the
$6Q$ remaining pairs, so the quad count is exactly $Q$. The base thus
has precisely $m$ rows and $F=z(m,n)$ cells, and is maximum.
The $15B$ bulk quads are equally divided among the three types.
Every local budget has bounded order and there are $O(n)$ groups, so
the local rows number $O(n)$; keep them all bare.
Lemma~\ref{lem:frameaugmentation} proves the result.
\end{proof}

\begin{theorem}\label{thm:criticallinear}
For every fixed $D>0$, uniformly for sufficiently large $n$ and
$|m-M/6|\le Dn$, one has $\zl(m,n)=mn/3+O_D(m)$.
\end{theorem}
\begin{proof}
The ranges of Lemmas~\ref{lem:criticalknown} and
\ref{lem:criticalstrips} cover all integer row counts in the window.
\end{proof}

\begin{lemma}\label{lem:quantitativecritical}
There are absolute $\varepsilon,C_0,n_0>0$ such that, for $n\ge n_0$
and $2n\le d=M/6-m\le\varepsilon n^2$,
\[
 \left|\zl(m,n)-mn/3\right|\le C_0(m+nd).
\]
\end{lemma}
\begin{proof}
Choose a fixed integer $K$ sufficiently large for the dense $K_4$ and
classical order thresholds, with $K\ge10^6$ and $49/(3K)<\delta_4$.
Choose $\varepsilon<\min\{1/48,\delta_5/(100K)\}$ and then $n_0$
sufficiently large in terms of these constants.
Put $Y=(M-6m)/4=3d/2$, $f=Y/n\ge3$, and
\[
 g=20\lceil K(f+1)/20\rceil,\qquad K(f+1)\le g\le2K(f+1).
\]
Partition the indices into $q$ ordinary groups of size $g$ and one
central group $g_c\in[g,2g)$. Set $v=3g$ and $v_c=3g_c+t$.
The maximum group ratio is at most
\[
 \frac{2g}{s}\le\frac{24K(f+1)}n
 \le36K\varepsilon+24K/n<\delta_5.
\]
Thus the dense $K_5$ decomposition and untraded bulk exist, even when
$d/n^2$ is a fixed positive number.

Let $y=Yg/s$. Choose each ordinary target $Y_i$ as the lower or upper
multiple of three adjacent to $y$, rounding the number of upper choices
so that
\[
 |Y_i-y|\le3,\qquad
 \left|\sum_iY_i-qy\right|\le3/2.
\]
Consequently $Y_c=Y-\sum_iY_i$ satisfies
$|Y_c-Yg_c/s|\le3/2$. No integrality of $Y$ or $Y_c$ is assumed.

For an ordinary group prescribe five-row degrees $\rho+3k_x$, with
balanced nonnegative $k_x$ of total $(5Y_i-\rho v)/3$.
This is integral, and the targets satisfy
\[
 2fg\le Y_i\le6fg+3,\qquad R_i:=\max_x(\rho+3k_x)\le12f.
\]
In particular
$2R_i(1+4R_i)\le1176f^2\le2Kf^2\le Y_i$.
Criterion~\eqref{eq:prescribeddegrees} gives exactly $Y_i$ linear five-rows.
After removing their pairs and $K_t$ from $K_{v+t}$, the degree and edge
congruences are those of the $A_4$ fillers above. Missing degree is at
most $\max\{48f,59\}\le49f$, whose ratio to the order is at most
$49/(3K)<\delta_4$.
The order exceeds the fixed decomposition threshold. Thus $K_4$ completion
provides the integer counts
\[
 m_i=(P_i-4Y_i)/6,\qquad F_i=(P_i+10m_i)/4.
\]

For the center, the bounded total rounding error gives
\[
 2fg_c\le Y_c\le7fg_c,\qquad
 Y_c/v_c\ge f/2,\qquad Y_c/v_c^2\le7/(9K).
\]
Its required row count $m_c=m-m_{\rm out}$ is integral and obeys
$m_c=M_c/6-2Y_c/3$. It lies in the actual classical range
\[
 \frac{M_c+((1-v_c)\bmod4)v_c}{10}+68
 \le m_c\le M_c/6-2\rho v_c/15-52.
\]
Indeed $2Y_c/3\ge v_c$ proves the upper inequality for large $v_c$.
For the lower inequality, use
$M_c/15\ge v_c^2/31$,
$2Y_c/3\le14v_c^2/(27K)$, and
$3v_c/10+68\le v_c^2/1000$; the chosen large $K$ makes the first
quantity exceed the sum of the other two. Also $v_c\ge12K$, so all
order thresholds hold. Choose a central maximum there.
The outside fixed count $(M-M_c+10m_{\rm out})/4$ is integral, so the
assembled base has exactly $m$ rows and $\lfloor(M+10m)/4\rfloor$ cells.
The universal convex bound proves it is maximum.

The number $L$ of bare rows satisfies
\[
 L\le\sum_i P_i/6+M_c/6\le ng\le3Kd+2Kn.
\]
Here the middle inequality follows from
$\sum g_i^2\le2gs$, $t<60$ and $g\ge60$.
The three typed classes have equal counts, and $m\ge n^2/24$ follows
from the chosen $\varepsilon$ and large $n$.
Lemma~\ref{lem:frameaugmentation} therefore gives
\[
 T\ge mn/3-O(m+nL)=mn/3-O(m+nd).
\]
The generic upper bound with
$F\le m+M=O(m)$ completes the proof.
\end{proof}

%% file: interior.tex
\subsection{All column residues in the interior}\label{sec:allinterior}
\begin{theorem}\label{thm:allinterior}
There are absolute constants $C,n_0$ such that $\zl(m,n)=mn/3+O(m)$
for every $n\ge n_0$ and every integer $n^2/12+Cn\le m\le n^2/6-Cn$.
\end{theorem}
\begin{proof}
Use the bounded-group frame of Section~\ref{sec:frame}, choosing $w$
sufficiently large once.

For an ordinary filler let $\epsilon=1$ if $n$ is even and $0$ otherwise.
Place $Q=\epsilon v/4$ disjoint quads on its old columns, then remove
these pairs and $K_t$ from $K_{v+t}$. Old residual degrees are
$v+t-1-3\epsilon$, new degrees are $v$, and the edge count is $P-6Q$.
These are respectively even and divisible by three. The missing degrees
are bounded, so dense triangle decomposition gives a filler with
\[
 m_o=P/3-Q,\qquad F_o=P-2Q=2m_o+P/3.
\]

Put $c_0=\lfloor M_c/4\rfloor$. For large fixed $w$, all five integers
$c_0,\ldots,c_0+4$ lie in the classical range
\[
 \frac{M_c+\beta v_c}{6}+40\le r
 \le\frac{M_c}{3}-\frac{\alpha v_c}{4}-28,
 \quad \beta=(1-v_c)\bmod3,\quad\alpha=(v_c-1)\bmod2.
\]
Thus \cite[Theorem 1.3]{chm-unbalanced} supplies central maxima with
$F_c=\lfloor M_c/3+2r\rfloor$ at these row counts.

Trade any $j\in\{0,\ldots,3B\}$ bulk groups, balanced among the three
color types, and choose $r=c_0+z$, $0\le z\le4$. The total row count is
\[
 m=15B+5j+qm_o+c_0+z.
\]
Consequently every integer from $a=15B+qm_o+c_0$ to $a+15B+4$ occurs.
Their endpoints are $n^2/12+O(n)$ and $n^2/6+O(n)$, uniformly in $n$,
which gives the theorem's interval after increasing $C$.

The outside pair budget is $P_{\rm out}=90B+qP=M-M_c$, and
$F_{\rm out}=2m_{\rm out}+P_{\rm out}/3$ is an integer.
Adding the central maximum therefore gives exactly
$\lfloor M/3+2m\rfloor$, the universal bound from
$\binom b2\ge3b-6$. The assembled base is maximum.
The $O(n)$ bare rows and typed counts within five allow
Lemma~\ref{lem:frameaugmentation}.
\end{proof}

%% file: even_transition.tex
\subsection{Even columns in a linear transition window}\label{sec:eventransition}
The classical value below is due to Chen, Horsley, and Mammoliti
\cite[Theorems 1.1--1.2]{chm-triple}. We construct maximum bases with
admissible augmentations.

\begin{theorem}\label{thm:eventransition}
Fix $D>0$. For every sufficiently large even $n$ and every integer $m$
with $|m-n^2/6|\le Dn$, one has
\[
 z(m,n)=\min\left\{
 \left\lfloor\frac{M+6m}{3}\right\rfloor,
 \left\lfloor\frac{6M+24m-n}{14}\right\rfloor,
 \left\lfloor\frac{M+3m}{2}\right\rfloor\right\},
 \qquad \zl(m,n)=\frac{mn}{3}+O_D(m).
\]
\end{theorem}
\begin{proof}
Use the bounded-group frame, with $w$ large in terms of $D$; here $t$ is even.
Fully trade the bulk, obtaining $30B$ typed triples, equally divided
among the types. Its pair budget is $90B$.

For an even order $x$ put
\[
 \ell_x=\frac{\binom x2}{3}-\frac x4,
 \qquad u_x=\frac{\binom x2}{3}+\frac x3.
\]
The classical formulas in the theorem apply in order below $\ell_x$,
between $\ell_x,u_x$, and above $u_x$. At an integral crossing the two
adjacent expressions agree. The central instances used below lie above
$\binom{x}{2}/6+x/3+40$ and below $\binom{x}{2}$ for sufficiently large
$x$, as required by the cited theorems. In particular they include the
crossings themselves, rather than the smaller interior range of
\cite[Theorem 1.3]{chm-unbalanced}.

Let $J=\lceil2(D+1)v\rceil$. The following ordinary fillers use $Q$ quad
rows, $R$ pair rows, and otherwise triples:
\[
\begin{array}{c|c|c|c|c}
 &Q&R&m_o&F_o\\ \hline
\text{lower}&v/4+j&0&P/3-v/4-j&P-v/2-2j\\
\text{middle}&3j&v/2-6j&P/3+v/3-7j&P+v/2-12j\\
\text{upper}&0&v/2+3j&P/3+v/3+2j&P+v/2+3j
\end{array}
\]
Here $0\le j\le J$ in the side rows and $0\le j\le v/12$ in the middle.
For the lower row prescribe quad incidences $1+2k_x$, with the $k_x$
nonnegative and differing by at most one, and
$\sum_x k_x=2Q-v/2$. For the upper row similarly prescribe pair
incidences with $\sum_xk_x=R-v/2$.
Their maximum degrees are $O_D(1)$ and their positive block counts are
at least $v/4$. Thus \eqref{eq:prescribeddegrees} realizes each family
when $v$ is sufficiently large. In the middle simply partition the old
points into the stated quads and pairs, since $4Q+2R=v$.

Remove these pairs and $K_t$ from $K_{v+t}$. The removed degree at every
old point is odd, and $6Q+R$ is divisible by three. The residual degrees
are therefore even, including the new degrees $v$, and its edge count
is divisible by three. Its missing degrees are $O_D(1)$, so dense triangle
decomposition proves all entries of the table.

The outside baseline $A_0=30B+qP/3=(M-M_c)/3$ is an integer.
In the middle, ordinary offsets fill a progression of step seven from
$-qv/4$ to $qv/3$. The central integer interval
$[\lceil\ell_{v_c}\rceil,\lfloor u_{v_c}\rfloor]$ has at least seven
members. Its addition fills exactly
$[\lceil\ell_n\rceil,\lfloor u_n\rfloor]$.
Below, use one central row count $\lfloor\ell_{v_c}\rfloor$;
above, use $\lceil u_{v_c}\rceil$ or $\lceil u_{v_c}\rceil+1$.
These give every integer respectively in
\[
 [\lfloor\ell_n\rfloor-qJ,\lfloor\ell_n\rfloor],
 \qquad
 [\lceil u_n\rceil,\lceil u_n\rceil+2qJ+1].
\]
All central values are the known adjacent-cap maxima.
The three intervals have no integer gaps. Since $qv\ge n/2$ eventually,
they cover the requested window.

For $P_{\rm out}=M-M_c$, the outside fixed counts in the three branches
are respectively the integers
\[
 2m_{\rm out}+P_{\rm out}/3,\qquad
 \frac{6P_{\rm out}+24m_{\rm out}-qv}{14},\qquad
 \frac{P_{\rm out}+3m_{\rm out}}2.
\]
Adding the corresponding central floor gives exactly the global classical
value. The middle expression counts only the $qv$ old ordinary columns;
the new columns are counted once in $v_c$. Thus every assembled base is
maximum. It has $O_D(n)$ bare rows, so
Lemma~\ref{lem:frameaugmentation} proves the augmented estimate.

For use in the high-row construction, at $m=\lceil(M+n)/3\rceil$ choose
$j=0$ in every upper filler and the center $\lceil u_{v_c}\rceil$.
Here $3\lceil u_{v_c}\rceil-M_c=v_c+2\epsilon$ with
$\epsilon\in\{0,1\}$, so the central A2 expression is integral.
Indeed $M_c+v_c=v_c(v_c+1)/2$ is congruent to zero or one modulo three.
Equality forces a pair-complete base with only pairs and triples.
The full base therefore has $n/2+\epsilon$ pair rows.
\end{proof}

%% file: odd_transition.tex
\subsection{Odd columns in a linear transition window}\label{sec:oddtransition}
\begin{theorem}\label{thm:oddtransition}
Fix $D>0$. For every sufficiently large odd $n$ and every integer $m$
with $|m-n^2/6|\le Dn$, one has $\zl(m,n)=mn/3+O_D(m)$.
\end{theorem}
\begin{proof}
We use the classical values of Chen--Horsley--Mammoliti
\cite[Theorems 1.1--1.2]{chm-triple}. In our row--column convention,
put $M_v=\binom v2$, $a_v=\lfloor M_v/3\rfloor$, and
\[
 C_v(r)=\min\left\{2r+a_v,\;
                    \left\lfloor\frac{M_v+3r}{2}\right\rfloor\right\}.
\]
For sufficiently large odd $v$ and
$M_v/6+v/3+40\le r\le M_v$, those theorems give
\begin{equation}\label{eq:oddclassicalcited}
 \begin{gathered}
 z(r,v)=C_v(r)-\delta_v(r-a_v),\\
 \delta_v(h)=
 \begin{cases}
 1,&v\equiv1,3\pmod6,\ h\in\{-4,-3,-2,-1\},\\
 1,&v\equiv5\pmod6,\ h\in\{-3,-1,0,1\},\\
 0,&\text{otherwise}.
 \end{cases}
 \end{gathered}
\end{equation}
This range contains the stated global window for large $n$, and all
bounded central offsets used below once their orders are large.

For later use, if $v\equiv5\pmod6$ and $r=a_v+3$, every central
maximizer attains the unrounded two/three cap:
\begin{equation}\label{eq:oddcentralequality}
 F=(M_v+3r)/2.
\end{equation}
Use $3a_v=M_v-1$ in \eqref{eq:oddclassicalcited}. The exact identity
$M_v+3r-2F=L+\sum(b-2)(b-3)/2$, where $L$ counts uncovered pairs,
then forces pair completeness and only two- and three-rows. There are
exactly $(3r-M_v)/2=4$ pair rows. Likewise at $v\equiv1,3\pmod6$,
$r=a_v$, equality forces an STS: the same identity gives only two/three
rows, and their pair-row count is zero.

Use the bounded-group frame, choosing $w$ large enough for the cited
central values and dense decompositions.
Fully trade the bulk, giving $10B$ triple rows of each type.
Here $t$ is odd; write $v=v_c$ for the central order.

For each of the $q$ ordinary groups, the graph $K_{3g+t}-K_t$ has even
degrees, edge count $P=\binom{3g}{2}+3gt$ divisible by three, and bounded
missing degrees. Decompose it into triples. Centrally choose a maximum
base at the required offset $-9\le h\le3$ in
\eqref{eq:oddclassicalcited}. All these filler rows are bare and number
$O(n)$. Their old columns lie in one index group, and their pair sets
are disjoint by the frame construction.

Set $M=\binom n2$, $a=\lfloor M/3\rfloor$, and $h=m-a$.
Before any bulk modification, the rows outside the central group are
all triples, with count
\[
 A=(M-M_v)/3=a-a_v
\]
and fixed-cell count $3A=M-M_v$. Since $n-v=3gq$, we have $v\equiv n\pmod6$.
The two caps and their exceptional losses therefore transfer exactly:
\[
 C_n(A+r)=3A+C_v(r),\qquad
 \delta_n(A+r-a)=\delta_v(r-a_v).
\]
We realize every requested row count as follows.

If $h=-k\le-5$, choose $k_c\in\{5,6,7,8,9\}$ congruent to $k$
modulo five and use a central maximum base with $a_v-k_c$ rows.
Its loss is zero. Untrade $(k-k_c)/5$ bulk local groups, balanced among
the three types. Each untrade replaces ten triples by five quads,
changing $(m,F)$ by $(-5,-10)$ and preserving the lower cap and zero
loss. For $h=-4,-3,-2,-1,0$, simply take the central maximum base with
$a_v+h$ rows, transferring its value exactly as above.

For $n\equiv1,3\pmod6$ and $h\ge1$, start with a central STS at
$a_v$ rows. Split $\lfloor h/2\rfloor$ bulk triples into their three
pair rows, balanced among types, and add one bare singleton in the remaining
parity. This attains $C_n(m)=\lfloor(M+3m)/2\rfloor$.
For $n\equiv5\pmod6$, handle $h=1$ by its central maximum base.
For $h\ge2$, choose $h_c=2$ or $3$ of the same parity as $h$, use the
central maximum at $a_v+h_c$, and split $(h-h_c)/2$ bulk triples.
Each split changes $(m,F)$ by $(2,3)$ and preserves the upper cap and
zero loss. Only $O_D(n)$ bulk groups or triples are modified, whereas
the supply is $\Theta(n^2)$. All these operations preserve $C_4$-freeness:
untrades and splits preserve their original covered-pair sets, disjoint
from those of the central base. The resulting fixed count equals
\eqref{eq:oddclassicalcited} globally, hence the base is exactly maximum.

Balancing leaves typed counts within five and $O(n)$ bare rows,
so Lemma~\ref{lem:frameaugmentation} gives the augmented estimate.
\end{proof}

%% file: high_rows.tex
\subsection{All larger row counts}\label{sec:highrows}
\begin{lemma}\label{lem:highrows}
For some absolute $C_1,n_1$,
$\zl(m,n)=mn/3+O(m)$ uniformly for $n\ge n_1$ and
$m\ge n^2/6+C_1n$.
\end{lemma}
\begin{proof}
Use the bounded index-group construction above, with a fixed sufficiently
large group size. There is a pair-complete initial base consisting only
of triangles and pairs, with equal typed row counts and $O(n)$ bare rows
whose old columns lie in one index group.
For even $n$, use the matching construction of
Section~\ref{sec:eventransition}, with no quad and
$R_*=n/2+\epsilon$, where $0\le\epsilon\le1$ makes $R_*\equiv M\pmod3$.
For $n\equiv1,3\pmod6$, use the all-triangle construction of
Section~\ref{sec:oddtransition}. For $n\equiv5\pmod6$, use the central
maximizer in \eqref{eq:oddcentralequality}, which has only pairs and
triangles, with $R_*=4$.
In every case the initial row count is
\[
 m_*=(M+2R_*)/3=n^2/6+O(n),\qquad m_*\equiv M\pmod2.
\]
All initial pair rows are bare, so the typed triangle counts are equal.

For $m_*\le m\le M$ with $m\equiv M\pmod2$, split
$(m-m_*)/2$ triangles into their three pair rows. Split typed triangles
as evenly as possible among the three types, keeping their pair rows
in the same type; after exhausting them, split bare triangles.
The number of available triangles is exactly $(M-m_*)/2$, so the process
reaches $m=M$. All pairs remain covered once, and the fixed-cell count
$(M+3m)/2$ attains the upper bound from $\binom b2\ge2b-3$.
For the opposite parity, use the preceding row count and add one typed
singleton fixed at an old column of its type, attaining
$\lfloor(M+3m)/2\rfloor$.

For $m\ge M$, start with the resulting all-pair base and append $m-M$
singleton rows. Assign their two-color types as evenly as possible.
Give each singleton any fixed label in its own two old colors.
This attains $z(m,n)=m+M$ by \eqref{eq:terminaldeficit}.
The bare rows still number $O(n)$, and the typed row counts have bounded
range at most three. Their fixed cells use only their two old colors,
and bare rows retain the single-group old supports.
Lemma~\ref{lem:frameaugmentation} gives $\zl(m,n)=mn/3+O(m)$
throughout the range, including all added singleton rows.
Choose $C_1,n_1$ large enough that the stated range begins above the
initial row count and its first opposite-parity successor.
\end{proof}

%% file: uniform_synthesis.tex
\subsection{Assembly of the threshold bounds}
\begin{proof}[Proof of Theorem~\ref{thm:sharpasymptotic}]
Let $C_0$ be the interior constant and $C_1$ the high-row constant.
The interior construction, followed by the odd and even transition
windows with a fixed constant larger than $\max\{C_0,C_1\}+1$, covers
all $m\ge n^2/12+C_0n$ with error $O(m)$.
Theorem~\ref{thm:criticallinear}, with $D\ge\max\{C_0+1,2\}$,
covers the remaining fixed linear band
down to $M/6-2n$ with the same error order.
Below it, Lemma~\ref{lem:quantitativecritical} applies, since
\[
 2n\le d=M/6-m\le\varepsilon n^2,
 \qquad d\le(n^2/12-m)_+.
\]
These constructions bound $mn/3-\zl(m,n)$ from above on every
$m\ge(1/12-\varepsilon)n^2$. Take $\varepsilon<1/48$.
For $m\le n^2/12$, Lemma~\ref{lem:linearcriticalgap} bounds this
difference from below by
$(\kappa m/n)(n^2/12-m)-O(m)$, which is at least
$(\kappa/16)n(n^2/12-m)-O(m)$.
For larger $m$, the generic upper bound supplies the lower estimate
$-O(m)$. Taking the largest fixed constants and thresholds proves the theorem.
For fixed $D$, its error is $O_D(m)$ whenever $m\ge n^2/12-Dn$;
at $m=\binom n2$ it gives $n^3/6+O(n^2)$.
\end{proof}

\begin{proof}[Proof of Theorem~\ref{thm:threshold}]
Theorem~\ref{thm:subcriticalgap} gives the conclusion for $0<c<1/12$.
For $c\ge1/12$, divide Theorem~\ref{thm:sharpasymptotic} by $mn$.
Its error tends to zero because
$1/n+(n^2/12-m)_+/m\to0$.
\end{proof}

%% file: deficit_profile.tex
\section{Sharp dependence on the fixed-base deficit}\label{sec:deficitprofile}
How much of the quarter guarantee survives when fixed cells are lost?
For $m/n^2\to\infty$, the relevant scale for the deficit $D$ is $n^2$.
The next theorem gives the sharp limiting guarantee as a function of
$d=\lim D/n^2$. Its upper examples use a large fixed row to restrict where
selected pairs can lie.

\begin{theorem}\label{thm:deficitprofile}
Suppose $n\to\infty$, $m/n^2\to\infty$, and nonnegative integers $D$
satisfy $D/n^2\to d$, where $0\le d<\infty$. Every $C_4$-free
$m\times n$ base with $|B|=z(m,n)-D$ satisfies
\[
 \frac{T^*(B)}{mn}\ge\Psi(d)-o(1),\qquad
 \Psi(d)=
 \begin{cases}
  1/4,&0\le d\le1/8,\\
  \sqrt{2d}-2d,&1/8\le d\le1/2,\\
  0,&d\ge1/2.
 \end{cases}
\]
The error is uniform over the base. For every such sequence of
dimensions and deficits, there are eventually bases of the prescribed size with
$T^*(B)/(mn)\to\Psi(d)$.
\end{theorem}

The proof uses the nonincreasing, $1$-Lipschitz function
$\varphi(p)=1/4-(p-1/2)_+^2$ on $[0,1]$.

\begin{lemma}\label{lem:largerowdeficit}
Let $m\ge n^2$, let $B$ be $C_4$-free, and put
$\Delta=m+M-|B|\le n^2$ and $b_{\max}=\max_r b_r$.
Then
\[
 T^*(B)\ge mn\varphi(b_{\max}/n)-O(m\sqrt n+n^3),
\]
with an absolute implicit constant.
\end{lemma}
\begin{proof}
By \eqref{eq:terminaldeficit}, at most $\Delta$ rows are empty,
and at most $M$ have degree at least two. Leave these rows bare;
there remain $G=m-O(n^2)$ singleton rows.

Let $g_y$ count singleton rows fixed at column $y$, and set
$H=\{y:g_y\ge m/\sqrt n\}$. Then $|H|\le\sqrt n$.
Keep $H$ and at most one additional column bare, leaving a column
set $K$ of even size $N=n-O(\sqrt n)$. We may assume $n$ is large;
otherwise the claimed lower bound is trivial after enlarging the constant.

Fix $y\in H$ and consider its $g_y$ singleton rows. The supports in
$K$ of fixed rows through $y$ are disjoint by $C_4$-freeness.
Use the nonempty supports as parts, adding singleton parts for uncovered
columns. Let their largest size be $k_y\le\min\{b_{\max},N\}$.
If $k_y\le N/2$, list the parts consecutively and match opposite
positions at distance $N/2$. No matched columns lie in the same part.
Pair the singleton rows arbitrarily and fill every resulting block
with a checkerboard pair. This gives $g_yN/4-O(N)$ pairs.

If $k_y>N/2$, let $S$ be the largest part and put $Q=K\setminus S$.
Split the rows into sets $X,Y$ with $|X|=\lfloor g_y|Q|/N\rfloor$.
Occupy $X\times S$ and $Y\times Q$, retain the same number of cells
from each rectangle, and pair the two pools by any bijection. This gives
at least $g_y k_y(N-k_y)/N-N$ pairs, including zero when $Q$ is empty.

Both constructions satisfy the original conditions in their simultaneous
use for all $y\in H$. Opposite cells are empty and selected cells are
disjoint. The common occupied columns of the two endpoint rows are
exactly $\{y\}$. Any row occupying $y$ is either a singleton row in
this group or a bare fixed row through $y$: other active groups have
different fixed anchors, and no selected cell is added in $H$.
In the balanced construction a group row occupies at most one column
of each matched pair, while a bare row has its $K$-support in one pencil
part. In the unbalanced construction every such row has its occupied
$K$-columns in only one of $S,Q$. Thus no row occupying $y$ contains both selected
columns, excluding every (C3) witness across all groups.

For $p=b_{\max}/n$, monotonicity and the Lipschitz bound give
\[
 N\varphi(k_y/N)
 \ge N\varphi(\min\{b_{\max}/N,1\})
 \ge n\varphi(p)-O(\sqrt n).
\]
Summing the heavy groups therefore gives at least
$\sum_{y\in H}g_y n\varphi(p)-O(m\sqrt n+n^{3/2})$ pairs.

Greedily pair the remaining light singleton rows with different fixed
columns. All unmatched rows then have one fixed column, so fewer than
$m/\sqrt n$ remain. Use any one column matching of $K$ for all these
row pairs. Skip blocks containing a fixed cell, at most two per row
pair, and add a checkerboard pair in every other block.
Each paired set of rows has an empty full common column neighborhood:
the fixed anchors are distinct, their blocks are bare, and every other
block is a checkerboard. Hence (C3) is impossible regardless of all
other groups; (S) and (C2) hold as well. These rows contribute at least
$G_{\rm light}n/4-O(m\sqrt n)$ pairs.

Since $\varphi\le1/4$, the two contributions total at least
$Gn\varphi(p)-O(m\sqrt n+n^{3/2})$.
Using $G=m-O(n^2)$ proves the lemma.
\end{proof}

\begin{proof}[Proof of Theorem~\ref{thm:deficitprofile}]
Eventually $m\ge M$, so \eqref{eq:terminaldeficit} has left side $D$ and gives
\[
 (b_{\max}-1)(b_{\max}-2)/2\le D,
 \qquad b_{\max}\le(3+\sqrt{1+8D})/2.
\]
For $d<1/2$, Lemma~\ref{lem:largerowdeficit} applies eventually;
its normalized error is $O(n^{-1/2}+n^2/m)=o(1)$.
Monotonicity and continuity of $\varphi$ give the required lower bound.
For $d\ge1/2$, nonnegativity suffices.

For sharpness, choose $b\in\{2,\ldots,n\}$ as large as possible with
$\binom{b-1}2=(b-1)(b-2)/2\le D$, and set $w=D-\binom{b-1}2$.
Choose $S\subseteq\{2,\ldots,n\}$ with $|S|=b-1$.
Take one blocker row with fixed support $\{1\}\cup S$,
one pair row for every column pair not contained in that support,
$w$ empty rows, and fill the remaining rows with singletons at column $1$.
There are $P=M-\binom b2$ pair rows and
$m'=m-P-1-w$ singleton rows. If $b<n$ then $w<b-1$; otherwise
$w=O(n^2)$. Thus $m'\ge0$ eventually and $m'/m\to1$.
Each column pair is covered exactly once, so this base is $C_4$-free,
and its size is
\[
 b+2P+m'=m+M-\binom{b-1}2-w=z(m,n)-D.
\]

In any augmentation, at most $O(n^3)$ selected pairs meet the blocker,
pair rows, or empty rows, by (S). Retain the other $T'$ pairs, with
$E=2T'$ cells in the $m'$ singleton rows and $N=n-1$ nonanchor columns.
Write $e_c$ for their selected column degrees. For each retained pair,
(C3) with witness column $1$ allows at most two common selected rows.
Thus $e_c+e_d\le m'+2$, including repeated endpoint lines, and
\[
 \sum_c e_c^2\le E(m'+2)/2.
\]
The blocker row and column $1$ also give an all-fixed (C3) witness
whenever both endpoint columns lie in $S$. Consequently
$E_S:=\sum_{c\in S}e_c\le E/2$.

Put $p=(b-1)/N$. If $p\le1/2$, Cauchy--Schwarz gives
$T'\le(m'+2)N/4$. If $1/2<p<1$ and $E>0$, the same inequality on the two
column sets gives, using $E_S/E\le1/2$,
\[
 \sum_c e_c^2\ge\frac{E_S^2}{pN}
                    +\frac{(E-E_S)^2}{(1-p)N}
             \ge\frac{E^2}{4p(1-p)N}.
\]
The last quadratic is minimized at $E_S/E=1/2$ on the allowed interval.
Hence $T'\le(m'+2)Np(1-p)$. The case $E=0$ is immediate; if $p=1$,
the blocker forbids every retained pair, so again $T'=0$.
Restoring the discarded pairs proves
\[
 T^*(B)\le O(n^3)+(m'+2)N\varphi(p).
\]
The choice of $b$ gives $b/n\to\min\{\sqrt{2d},1\}$.
Dividing by $mn$ yields the matching upper limit $\Psi(d)$.
Also $|B|\le m+M=o(mn)$, so the same profile holds for the original
objective $|B|+T^*(B)$ with the base fixed.

At $b=n$ and $w=0$, the base consists of one full row and singleton rows
at column $1$, and admits no selected pair. Its size
$m+n-1\sim z(m,n)$ shows that relative closeness to the classical
maximum alone gives no augmentation guarantee.
\end{proof}

%% file: square_lower.tex
\section{A finite square construction}\label{sec:squarelower}

\begin{theorem}\label{thm:squarelower}
Let $q=2^k$, where $k\ge1$ is an integer, and put $N=q^2+q+1$. Then
\[
 \zl(N,N)\ge N(q+1)+\frac{q^3(q-2)}4\ge\frac{N^2}{4}.
\]
\end{theorem}
\begin{proof}
Pair counting and Cauchy--Schwarz give $F(F-N)\le N^2(N-1)$,
attained by the projective plane over $\mathbb F_q$ with $F=N(q+1)$.
Write its affine points as $P_{a,b}$, finite-slope lines as
$L_{s,t}:b=sa+t$, vertical lines as $V_a$, and points at infinity as
$I_s,I_\infty$. For $\Delta\ne0$, the characteristic-two translations
$P_{a,b}\leftrightarrow P_{a,b+\Delta}$ and
$L_{s,t}\leftrightarrow L_{s,t+\Delta}$ give disjoint row and column
pairs, taking one pair per orbit. Leave all points at infinity bare.
The common fixed column of each row pair is $V_a$; the common row
anchor of a slope-$s$ column pair is $I_s$, which is nonincident with
$V_a$. Thus the filter in Lemma~\ref{lem:bareanchors} always passes.
For each row pair and slope, precisely one column pair contains the
fixed intercepts $b-sa,b-sa+\Delta$; all other blocks are fixed-cell-free.
The lemma gives
\[
 T=(q^2/2)q(q/2-1)=q^3(q-2)/4.
\]
Finally,
\[
 N(q+1)+q^3(q-2)/4-N^2/4=(5q^2+6q+3)/4>0.
\]
\end{proof}

%% file: polynomial_certificate.tex
\section{Polynomial certificate for the separation bound}\label{sec:certificate}
We verify the polynomial inequalities used in Lemma~\ref{lem:lightmatrix}.
The argument here depends only on the displayed coefficients, not on
the conclusion of that lemma. Let $G,P$ be the polynomials defined in
its proof, with $\lambda=-81/100$.

Write $\mathcal B_{i,d}(x)=\binom di x^i(1-x)^{d-i}$. Define
\[
 a(x)=\frac1{100}\sum_{i=0}^7 A_i\mathcal B_{i,7}(x),\qquad
 b(x)=\frac1{100}\sum_{i=0}^7 B_i\mathcal B_{i,7}(x)
\]
from the table, and let $h$ have degree-five Bernstein coefficients
$(-3669,-430,0,0,430,3669)/100$.
These coefficients are nondecreasing and antisymmetric, so $h$ is
nondecreasing and $h(1-x)=-h(x)$. 
\begin{center}
\small
\begin{tabular}{c|rrrrrrrr}
 $i$&0&1&2&3&4&5&6&7\\ \hline
 $A_i$&10000&-6978&9672&-7973&4652&-1386&-701&-843\\
 $B_i$&-69&-7346&7198&-9370&6685&-4103&-2529&-2204
\end{tabular}
\end{center}

For completeness, the finite verification uses only rational Bernstein
coefficients. Elevating a coefficient vector $(c_j)_{j=0}^d$ to degree
$D\ge d$ gives
\[
 \widehat c_i=
 \sum_j c_j\frac{\binom dj\binom{D-d}{i-j}}{\binom Di},
\]
with out-of-range binomial coefficients zero. At degree eight, $xa(x)$
has coefficient zero at index zero and $iA_{i-1}/800$ at index $1\le i\le8$.
The coefficients of $x,x^2$ are $i/8,i(i-1)/56$.
These formulas and tensor products give the degree-$(8,8)$ coefficients
of $G,P$. Repeated midpoint averaging subdivides a Bernstein polynomial
into its two half-interval representations; apply this alternately on
the two axes until every coefficient is nonnegative.
The exact checker \texttt{anc/verify\_joint\_bound.py} terminates with
$35$ rectangles for $G$ and $28$ for $P$, checking $5103$ final
coefficients. Each subdivision preserves full coverage of the closed
unit square. The Bernstein basis is nonnegative and sums to one, so
the coefficient checks prove both inequalities on the closed unit square $[0,1]^2$.
The checker contains the displayed integer data and uses only exact
fractions; it has no optimizer or floating-point dependency.

%% file: cubic_growth.bbl
\begin{thebibliography}{99}
\small
\setlength{\itemsep}{2pt}
\raggedright
\bibitem{framework} L. Qi, C. Cui, and Y. Xu,
\emph{Biquadratic SOS rank and augmented Zarankiewicz number},
arXiv:2603.04912v4 (2026).
\url{https://arxiv.org/abs/2603.04912v4}.
Published in Mathematics 14 (2026), no.\ 9, Article 1552,
\url{https://doi.org/10.3390/math14091552}.
Definitions and numbering here follow the cited arXiv version.
\bibitem{general} L. Qi, C. Cui, and Y. Xu,
\emph{A general lower bound for the limited augmented Zarankiewicz number
based upon complete graphs}, arXiv:2604.04111v5 (2026).
\url{https://arxiv.org/abs/2604.04111v5}.
\bibitem{incidence} Y. Xu, G. Yu, and L. Qi,
\emph{Optimization models and computational bounds for limited augmented
Zarankiewicz numbers in the incidence-graph family of complete graphs},
arXiv:2605.29658v2 (2026).
\url{https://arxiv.org/abs/2605.29658v2}.
\bibitem{k5} H. Liu and Y. Song,
\emph{New bounds for limited Zarankiewicz numbers from $K_{5t}$ blocks},
arXiv:2609.07227v1 (2026).
\url{https://arxiv.org/abs/2609.07227v1}.
\bibitem{designexistence} D. Horsley and D. A. Pike,
\emph{On balanced incomplete block designs with specified weak chromatic number},
J. Combin. Theory Ser. A 123 (2014), 123--153.
\url{https://doi.org/10.1016/j.jcta.2013.12.004}.
Author preprint: arXiv:1209.6111v2.
\bibitem{brouwer} A. E. Brouwer,
\emph{The linear spaces on 15 points}, Ars Combin. 12 (1981), 3--35.
Preprint ZW38/79 (1979), \url{https://ir.cwi.nl/pub/6907/6907D.pdf}.
\bibitem{densedecomp} B. Barber, D. K\"uhn, A. Lo, and D. Osthus,
\emph{Edge-decompositions of graphs with high minimum degree},
Adv. Math. 288 (2016), 337--385.
\url{https://doi.org/10.1016/j.aim.2015.09.032}.
Author preprint: arXiv:1410.5750v3.
\bibitem{llf} Z. Lv, M. Lu, and C. Fang,
\emph{A note on 3-partite graphs without 4-cycles},
J. Combin. Des. 28 (2020), 753--757.
\url{https://doi.org/10.1002/jcd.21742}.
\bibitem{chm-unbalanced} G. Chen, D. Horsley, and A. Mammoliti,
\emph{Exact values for some unbalanced Zarankiewicz numbers},
J. Graph Theory 106 (2024), 81--109.
\url{https://doi.org/10.1002/jgt.23068}.
Author preprint: arXiv:2202.05507v2.
\bibitem{chm-triple} G. Chen, D. Horsley, and A. Mammoliti,
\emph{Zarankiewicz numbers near the triple system threshold},
J. Combin. Des. 32 (2024), 556--576.
\url{https://doi.org/10.1002/jcd.21948}.
Author preprint: arXiv:2310.12685v2.
\bibitem{weakincidence} J. L\"ofberg and L. Qi,
\emph{Second order Zarankiewicz number}, arXiv:2608.30555v2 (2026).
\url{https://arxiv.org/abs/2608.30555v2}.
\bibitem{llf2} Z. Lv, M. Lu, and C. Fang,
\emph{Density of balanced 3-partite graphs without 3-cycles or 4-cycles},
Electron. J. Combin. 29 (2022), no. 4, Paper No. 4.44.
\url{https://doi.org/10.37236/10958}.
\end{thebibliography}
